\documentclass[12pt]{amsart}

\usepackage{graphicx}
\usepackage{amssymb, comment, color}
\usepackage[toc,page]{appendix}
\usepackage{epsf,overpic,amssymb,amscd,times,euscript}
\usepackage[urlcolor=blue,citecolor=blue]{hyperref}
\usepackage[nocompress]{cite}

\usepackage [autostyle, english = american]{csquotes}
\MakeOuterQuote{"}
\usepackage[utf8]{inputenc}
\usepackage[english]{babel}
\usepackage{amsmath}
\usepackage{amsthm}
\usepackage{amsfonts}
\usepackage{amssymb}
\usepackage{amscd}
\usepackage{bezier}
\usepackage{enumerate}
\usepackage{tikz}
\usepackage{graphicx}
\usepackage{hyperref}
\usepackage{upgreek}

\makeatletter
\newcommand{\dtilde}[1]{{%
  \mathpalette\double@widetilde{#1}%
}}
\newcommand{\double@widetilde}[2]{%
  \sbox\z@{$\m@th#1\widetilde{#2}$}%
  \ht\z@=.9\ht\z@
  \widetilde{\box\z@}%
}
\makeatother

\hypersetup{colorlinks}

\usepackage{parskip}

\usepackage{bbm}

\newcommand{\id}{id}

\newcommand{\N}{\mathbbm{N}}                     
\newcommand{\Z}{\mathbbm{Z}}                     
\newcommand{\R}{\mathbbm{R}}                     

\newcommand{\Fix}{\mathrm{Fix}}

\newcommand{\dist}{\mathrm{dist\,}}             
\newcommand{\CP}{\mathbbm{CP}}                   
\newcommand{\Ham}{\mathrm{Ham}}    
\newcommand{\HF}{\mathrm{HF}}    
\newcommand{\loc}{\mathrm{loc}} 
\newcommand{\CF}{\mathrm{CF}}
\newcommand{\im}{\mathrm{im}}

\newcommand{\dom}{\mathrm{dom}}

\newcommand{\topo}{\mathrm{top}}

\newcommand{\hofer}{\mathrm{Hofer}}

\newcommand{\vol}{\mathrm{vol}}

\newtheorem{mainthm}{\sc Theorem}           
\newtheorem{maincor}
{\sc Corollary}

\newtheorem{thm}{Theorem}[section]               
\newtheorem*{thm*}{Theorem}               
\newtheorem{cor}[thm]{Corollary}        
\newtheorem*{cor*}{Corollary}        
\newtheorem{lem}[thm]{Lemma}  
\newtheorem*{lem*}{Lemma}
\newtheorem{prop}[thm]{Proposition}     
\theoremstyle{definition}
\newtheorem{rem}[thm]{Remark}           
 \newtheorem*{acknowledgement*}
 {\protect\acknowledgementname}
 \newtheorem*{qu*}{Question}
\newcounter{claim}
\newenvironment{claim}[1][]{\refstepcounter{claim}\par\noindent\underline{Claim~\theclaim:}\space#1}{}
\newenvironment{claimproof}[1]{\par\noindent\underline{Proof:}\space#1}{\qed}

 \providecommand{\acknowledgementname}{Acknowledgement}

\author{Matthias Meiwes}
\thanks{The author was supported by the ERC Starting Grant 757585 and the Israel Science Foundation Grant 938/22}

\address{Matthias Meiwes,
School of Mathematical Sciences, Tel Aviv University, Ramat Aviv, Tel Aviv 69978, Israel.}
\email{\texttt{matthias.meiwes@live.de}}

\title{On the barcode entropy of Lagrangian submanifolds} 
\begin{document}
\maketitle
\begin{abstract}
This article deals with relative barcode entropy, a notion that was recently introduced by \c{C}ineli, Ginzburg, and G\"urel. We exhibit some 
 settings in closed symplectic manifolds for which the relative barcode entropy of a Hamiltonian diffeomorphism and a pair of Lagrangian submanifolds is positive.  In analogy to a result in the absolute case by the above authors, we obtain that the topological entropy of any horseshoe $K$ is a lower bound if the two  Lagrangians contain a local unstable resp.\  stable manifold in $K$. In dimension~$2$, we also estimate the relative barcode entropy of a pair of closed curves that lie in special homotopy classes in the complement of certain periodic orbits in $K$. Furthermore, we define a variant of relative barcode entropy  and exhibit first examples for which it is positive. As applications, certain robustness features of the volume growth and the topological entropy are discussed.  
\end{abstract}
\section{Introduction and main results}
To study the asymptotic dynamics properties of a diffeomorphism $\psi$ on a manifold $M$, one might look at the growth of specific quantities that are assigned to its iterates $\psi^k$. As a natural example,  we can consider the growth of volumes of submanifolds under iterations: we can fix a submanifold $N$ in $M$ and consider the growth of volumes of $\psi^k(N)$; by looking at growth specifics that are independent of the metric, we might obtain some interesting dynamical quantities of $\psi$ "relative" to $N$. Moreover, by simultaneously considering $N$ in a suitable family of submanifolds,  this might yield some dynamical invariants assigned to $\psi$ itself. Especially interesting is the situation when there is exponential volume growth: due to results by Yomdin, \cite{Yomdin87}, and Newhouse, \cite{Newhouse88}, the supremum $\sup_{N\subset M}\Gamma(\psi;N)$ of the exponential volume growth rates \begin{align}\label{Gamma}\Gamma(\psi;N) := \limsup_{k\to \infty} \frac{\log (\vol(\psi^k(N)))}{k}
 \end{align}
 over all submanifolds $N$,  
is identical to the topological entropy, $h_{\topo}(\psi)$, of~$\psi$, provided $\psi$ is $C^{\infty}$ smooth. 
If $\psi$ is Hamiltonian, a related, though more abstract  notion is relative  barcode entropy, $\hbar(\psi;L_0,L_1)$, defined for suitable pairs $(L_0,L_1)$ of Lagrangian submanifolds, introduced and first studied by \c{C}ineli, Ginzburg, and G\"urel in \cite{CGG_barcode}. 
Relative barcode entropy provides a lower bound on volume growth, that is,  $\hbar(\psi;L_0,L_1) \leq \Gamma(\psi;L_0)$, \cite{CGG_barcode}.
In this article, we will discuss some natural settings for which $\hbar(\psi;L_0,L_1)$  is bounded from below. But first, let me give some further background and motivation.  

In the symplectic setting, various results have been obtained that deal with asymptotic dynamics features that are induced only by the  symplectic topology of the underlying manifold, together with maybe some mild additional assumptions on the diffeomorphisms. Several  results in this direction can be found in Polterovich' paper \cite{Polterovich_growth}. As an example dealing with volume growth, it holds that for any non-identical Hamiltonian diffeomorphism on a closed surface with positive genus there exists a curve $N$ such that the length of $\psi^k(N)$  grows at least linearly (for the case of the $2$-torus, see also \cite{Polterovich-Sikorav}).   Subsequently, various results of this type, concerning polynomial and exponential growth of Lagrangian submanifolds,  were obtained by Frauenfelder and Schlenk, \cite{FrauenfelderSchlenk2005,FrauenfelderSchlenk2006}. 
Topological entropy in this context 
has been 
studied in the contact setting, first for unit cotangent bundles in   \cite{MacariniSchlenk2011},   and 
 then in a variety of situations by  Alves and other authors, see     
\cite{AASS, Alves1, A2, Alves-Anosov, AlvesColinHonda, AlvesMeiwes, Meiwesthesis, Dahinden2018}. In one form or the other, an important part of the reasoning in the works above is that the complexity of some type of Floer homology or contact homology on the underlying manifold serves as a  "lower bound" on the complexity of the dynamics. The underlying theory that one considers might be  "relative"   (generators correspond to chords between Lagrangian/Legendrian submanifolds), as e.g.\ in \cite{MacariniSchlenk2011,AlvesMeiwes,AlvesColinHonda},  or  "absolute"  (generators correspond to periodic orbits), as e.g.\ in \cite{Alves1, Meiwesthesis}.

As discovered in \cite{CGG_barcode}, and 
 investigated further in \cite{CGG_subexp_barcode, Mazzucchelli_barcode, beomjun_barcode, CGGM_barcode},  
this general line of reasoning can be partly reversed: certain complexity features of the dynamics of a diffeomorphism  translate into growth of complexity of Floer data assigned to its  iterations, into \textit{barcode entropy}. 
This notion, introduced in \cite{CGG_barcode} and which has again a relative and an absolute version,  is defined naturally in the framework of persistence homology (for persistence homology in Floer theory see \cite{Usher_boundary_depth, polterovich2016autonomous,Usher_Zhang}  or the book \cite{polterovich2019topological}).
One assigns to a filtered Floer chain complex  a barcode, a multiset of intervals ("bars"), where, roughly speaking,  each bar captures a homology class that "persists" in the filtered homology  for a certain filtration window, and the length of the bar corresponds to the length of that filtration window. 
Barcode entropy then measures the exponential  growth rate of the number of bars, ignoring bars of very small size, for a  sequence of Floer chain complexes associated  to iterations of the diffeomorphism. 
More precisely, the \textit{relative} barcode entropy $\hbar(\psi;L_0,L_1)$ of a triple $(\psi,L_0,L_1)$, where $\psi$ is a Hamiltonian diffeomorphism on a symplectic manifold $M$, and 
$(L_0,L_1)$ 
 is a  pair of  Hamiltonian  isotopic,   monotone Lagrangian submanifolds in $M$, is defined as 
$$ \hbar(\psi;L_0,L_1) = \lim_{\epsilon\to 0} \limsup_{k\to \infty}\frac{ \log(b_{\epsilon}(\psi;L_0,L_1, k))}{k},$$
where $b_{\epsilon}(\psi;L_0,L_1, k)$ is the number of bars of length at least  $\epsilon$ in the Lagrangian Floer homology of $(L_0,\psi^{-k}(L_1))$\footnote{This definition makes sense if the pair $(L_0,L_1)$ is \textit{non-degenerate}, that is, if $L_0$ intersects $\psi^{-k}(L_1)$ transversally for all $k$. One can then naturally extend the definition to degenerate pairs, see \S \ref{sec:relbarent}.}. 
Analogously, by considering the Hamiltonian Floer chain complexes of $\psi^k$, one defines the \textit{absolute} barcode entropy $\hbar(\psi)$ of Hamiltonian diffeomorphisms $\psi$. 
Now,  the following lower bound on the \textit{absolute} barcode entropy $\hbar(\psi)$ in terms of the complexity of the dynamics of a Hamiltonian diffeomorphism $\psi: M\to M$ is obtained in \cite{CGG_barcode}.  Given that  $\psi$ has a compact hyperbolic invariant set $K$ in $M$  which is locally maximal, then $\hbar(\psi)$ is bounded from below by the topological entropy on $K$, \cite[Thm.~B]{CGG_barcode}, that is, 
\begin{align}\label{bar} \hbar(\psi)\geq h_{\topo}(\psi|_K). 
\end{align}
On the other hand,   some reverse inequalities are still satisfied.  It holds  that 
\begin{align}\label{bar_ent}\hbar(\psi;L_0,L_1) \leq \Gamma(\psi;L_0) \leq h_{\topo}(\psi)
    \end{align}
(as long as $\hbar(\psi;L_0,L_1)$  is defined),  \cite[Thm.\ A]{CGG_barcode}, and similarly,   $\hbar(\psi)$ bounds  $h_{\topo}(\psi)$ from below. 
In dimension $2$, \cite[Thm.\ C]{CGG_barcode}, this 
leads to the  intriguing identity   \begin{align}\label{equal_bar}\hbar(\psi) = h_{\topo}(\psi).
\end{align}

\color{black}
Inequalities \eqref{bar} and \eqref{bar_ent} give rise to the following circle of questions about relative barcode entropy: Are  there, in view of \eqref{bar} for the absolute case, bounds that are  reverse to that in \eqref{bar_ent}? For instance, given any  $\psi$ with  $h_{\topo}(\psi)>0$, does there exist  a pair of Lagrangian submanifolds $(L_0,L_1)$ such that $\hbar(\psi;L_0,L_1)>0$?  Or, a more approachable question: Assume we are in a setting where it is known that the volume growth $\Gamma(\psi;L)$ is positive, can we find for the same $\psi$ also pairs $(L_0,L_1)$ for which $\hbar(\psi;L_0,L_1)>0$?

Apart from the fact that those questions arise naturally given the inequalities \eqref{bar} and \eqref{bar_ent}, some motivation for them comes from stability properties of barcode entropy with respect to the Hofer-distance and the $\gamma$-distance on Hamiltonian isotopy classes of Lagrangian submanifolds: relative barcode entropy $\hbar(\psi;L_0,L_1)$, as a function in $L_0$ and $L_1$ varying in their Hamiltonian isotopy class, is lower semi-continuous with respect to those distances. And hence, in  view of the first inequality in \eqref{bar_ent},  lower estimates for relative barcode entropy might imply some  interesting robustness statements for the volume growth of Lagrangians (see \S \ref{sec:intro:stability} for further discussion). 

Some examples for which the relative barcode entropy is positive can already be derived directly from \eqref{bar}: by a well-known identification of Hamiltonian Floer homology in $M$ and Lagrangian Floer homology in the  product $M\times M$, one obtains the identity 
$\hbar(\psi) =   \hbar(\id \times \psi; \Lambda, \Lambda)$, where $\id \times \psi: M \times M \to M \times M$ is the graph map and $\Lambda \subset M \times M$ is the diagonal (cf.\ \cite{CGG_barcode}).   
Examples obtained via that identification are quite specific. In particular, this procedure does not yield examples on manifolds different from symplectic products, and it is for instance not at all clear from the above whether there exist even one surface diffeomorphism $\psi$ and two closed curves $L_0,L_1$  for which $\hbar(\psi;L_1,L_2)>0$.

Let us now state the main results of this article. 
 Throughout the article, we assume that the Lagrangian submanifolds  are monotone with first Chern number at least $2$, cf.\ \S \ref{sec:preliminaries}. Moreover, we assume that the Lagrangian submanifolds $L_0$ and $L_1$ when dealing with  $\hbar(\psi;L_0,L_1)$ are Hamiltonian isotopic. 
  We say that such pairs $(L_0,L_1)$  are \textit{admissible}.

\subsection{Barcode entropy for Lagrangian submanifolds that contain local stable and unstable disks} 
To formulate the first theorem, we recall that a hyperbolic invariant set $K$ of a diffeomorphism $\psi$ is \textit{locally maximal} if there exists a compact neighbourhood $U$ for which $K= \bigcap_{k\in \Z} \psi^k({U})$; such a neighbourhood is called \textit{isolating}. For  every point $q$  in $K$ and $\delta>0$, we 
 have \textit{local stable and unstable manifolds of size $\delta$},  $W^s_{\delta}(q)$ and  $W^u_{\delta}(q)$, see~\S \ref{sec:stable}, which, for small enough  $\delta$, are small Lagrangian disks 
passing through~$q$. The following might be considered as a relative version of inequality~\eqref{bar_ent}, that is, of Theorem~B in~\cite{CGG_barcode}. 
\begin{mainthm}\label{mainthm:higherdim}
Let $K$ be a locally maximal, topologically transitive  hyperbolic set.  If $(L_0,L_1)$ is an admissible pair  of Lagrangian submanifolds  in $M$ with   $W_{\delta}^u(q) \subset L_0 \text{ and } W_{\delta}^s(p) \subset L_1 \text{ for some }q,p, \in K, \delta>0$, then $$\hbar(\psi;L_0,L_1) \geq h_{\topo}(\psi|_K).$$
\end{mainthm}
Note that every locally maximal compact hyperbolic set contains a  topologically transitive one with the same  topological entropy. Note also that the condition to contain 
 $W_{\delta}^u(q)$ or $W_{\delta}^s(p)$ is a local condition. In fact, given $K$ as above,  it is possible to modify  any Lagrangian $L$ to one for which this condition is satisfied, keeping all but a small set inside $L$ fixed.

Together with \cite[Thm.~A]{CGG_barcode}, cf.\ \eqref{bar_ent}, and a well-known result by Katok about the approximation of topological entropy by that of  horseshoes in dimension~2,  \cite{KatokLyapunov_exp},  Theorem  \ref{mainthm:higherdim} provides the following result; it resembles the Yomdin-Newhouse identity (in dimension $2$). 
\setcounter{maincor}{1}
\begin{maincor}\label{cor:hbar=htop} If $\psi:M\to M$ is a Hamiltonian diffeomorphism on a closed symplectic surface, then 
$$\overline{\hbar}(\psi):=\sup \hbar(\psi;L_0,L_1) = h_{\topo}(\psi),$$
where the supremum runs over all admissible pairs of closed curves $(L_0,L_1)$ in $M$.
\end{maincor}
\begin{rem}\label{rem:lsc}
It was shown in \cite{AlvesMeiwesBraids} that, if $M$ is a closed surface, the topological entropy $h_{\topo}(\psi)$ is lower semi-continuous with respect to the Hofer-distance on $\Ham(M,\omega)$. Hence, by Corollary~\ref{cor:hbar=htop} the same holds for the quantity $\overline{\hbar}(\psi)$, cf. \S \ref{sec:intro_strong}.  
\end{rem}
\begin{rem}
It would be interesting to understand whether also in higher dimensions, positive topological entropy   ($h_{\topo}(\psi)>0$)  implies that there exist pairs $(L_0,L_1)$ with $\hbar(\psi;L_0,L_1)>0$. The analogous statement for the absolute barcode entropy does not hold as examples of \c{C}ineli show, \cite{Cineli_example}.
\end{rem}
\subsection{Barcode entropy of closed curves in the complement of some horseshoe orbits}
To motivate the next result, let me recall the following fact for surface diffeomorphisms with positive topological entropy. In general, for any  continuous map $f:M\to M$ on a compact surface $M$ and any set of periodic orbits $\mathcal{O}\subset M$ of $f$, we obtain an induced action on the fundamental group in the complement $f_{\#}: \pi_1(M\setminus \mathcal{O})\to \pi_1(M,\setminus \mathcal{O})$ (up to some choices).  Furthermore, given a homotopy class $\alpha$ of closed curves in $M\setminus \mathcal{O}$, we can consider the exponential growth rate $\Gamma_{\mathcal{O}, \alpha}$ of $f^n_{\#}(\alpha)$ in  
$\pi_1(M\setminus \mathcal{O})$, in terms of the minimal length of words of generators necessary to represent $f^n_{\#}(\alpha)$, see \cite[\S 3.1.c]{Hasselblatt-Katok}.  
This growth rate is a lower bound, provided that $f$ is Lipschitz, to the growth rate of the length of closed curves representing~$\alpha$. That is, for any closed curve $L$ with $[L] = \alpha$ we have that 
$$\Gamma(f;L)\geq \Gamma_{\mathcal{O},\alpha}(f).\footnote{In fact, $\inf_{[L]=\alpha} \Gamma(f;L) = \Gamma_{\mathcal{O},\alpha}(f)$.}$$
Let $f$ be now a $C^{1+\alpha}$-diffeomorphism with $h_{\topo}(f)>0$. By a result of Franks and Handel, \cite{FranksHandel1988}, there exist a periodic orbit $\mathcal{O}$ and a homotopy class $\alpha$ of curves in the complement such that $
\Gamma_{\mathcal{O},\alpha}(f)>0$. 
This result was improved by Alves and the author, \cite{AlvesMeiwesBraids}: For every $0<e<h_{\topo}(f)$, there exist $\mathcal{O}$ and $\alpha$ as above such that $$\Gamma_{\mathcal{O},\alpha}(f)>h_{\topo}(f)-e.$$ 
In particular, for any closed curve $L$ with $[L] = \alpha$, 
$$\Gamma(f;L) > h_{\topo}(f)-e.$$
The following result states that if we allow $\mathcal{O}$ to be a finite union of periodic orbits, then we can replace in that result the exponential length growth rate  by relative barcode entropy.
 \setcounter{mainthm}{2}
\begin{mainthm}\label{thm:barcode_lagr}
 Let $\psi:M\to M$ be a Hamiltonian diffeomorphism with $h_{\topo}(\psi)>0$  on a closed symplectic surface $(M,\omega)$. Let $0<e<h_{\topo}(\psi)$. 
 Then there exist a finite invariant set $\mathcal{O}\subset M$  and homotopy classes $\alpha_0,\alpha_1$ of closed curves in $M\setminus \mathcal{O}$ such that for any admissible pair $(L_0,L_1)$ with $[L_i] = \alpha_i$, $i=0,1$, 
 $$ \hbar(\psi;L_0,L_1) > h_{\topo}(\psi) -e.$$
 \end{mainthm}

\subsection{Stability with respect to the Hofer- and the $\gamma$-distance}\label{sec:intro:stability} 
Two fundamental and remarkable norms defined on groups of Hamiltonian diffeomorphisms are the Hofer-norm $\|\cdot\|_{\hofer}$ and the spectral- (or $\gamma$-)norm~$\gamma(\cdot)$. The former was first introduced by Hofer, \cite{Hofer90}, whereas the latter was introduced by Viterbo, \cite{Viterbo92}, by means of generating functions, and by Schwarz, see \cite{Schwarz2000}, and Oh, see \cite{Oh05}, using Floer theory. The Hofer-norm bounds the $\gamma$-norm from above, whereas the $C^k$ norms are not bounded by those norms in general.  We refer to \S \ref{sec:relbarent} and to the above references for definitions of these norms. They naturally induce distances on the groups of Hamiltonian diffeomorphisms, and furthermore also on any given Hamiltonian isotopy class $\mathcal{L}$ of Lagrangian submanifolds: 
the \textit{Hofer-distance} and (exterior) \textit{$\gamma$-distance} on $\mathcal{L}$ is defined as 
 $$d_{\hofer}(L,L') :=\inf\{\|\psi\|_{\hofer}\,|\, \psi \in \Ham(M,\omega), \psi(L) = L'\}, $$
$$d_{\gamma}(L,L') : = \inf \{\gamma(\psi)\, |\, \psi \in \Ham(M,\omega), \psi(L) = L'\},$$
see also  \cite{Kislev-Shelukhin} and references therein.

A remarkable property of the relative barcode entropy is that 
$\hbar(\psi;L_0,L_1) $ is lower semi-continuous with respect to $d_{\hofer}$ and $d_{\gamma}$  in the pair $(L_0,L_1)$ varying in their Hamiltonian isotopy class. 
Together with these properties and estimate \eqref{bar_ent},  Theorem~\ref{mainthm:higherdim}
has the following consequence. 
\setcounter{maincor}{3}
\begin{maincor}\label{cor:Gamma_stable}
Let $\psi$ be a Hamiltonian diffeomorphism on a closed symplectic manifold $(M,\omega)$,  and let  $\mathcal{L}$ be a Hamiltonian isotopy class of monotone Lagrangian submanifolds in $M$. Assume that $L_0$ in $\mathcal{L}$ 
 contains a local unstable manifold of a point in a locally maximal compact hyperbolic set $K$ with $h_{\topo}(\psi|_K)>0$. Then, there exists a $d_{\gamma}$-open (and hence also $d_{\hofer}$-open) neighbourhood $\mathcal{U}$ of $L_0$ in $\mathcal{L}$ such that for all $L \subset \mathcal{U}$, 
$$\Gamma(\psi;L)>0.$$
\end{maincor}
\begin{rem}
The statement is of rather symplectic nature and it is  not clear if it can be obtained without using methods from Floer theory or related methods. 
To illustrate that, we might just consider Lagrangians $L$ which are graphs of exact $1$-forms $df$  under the identification of a neighbourhood of $L_0$ with the unit co-disk bundle of $L_0$: In that case, $d_{\hofer}(L,L_0)$ and $d_{\gamma}(L, L_0)$ are bounded from above by $\max f - \min f$, a quantity that does not control large oscillations of $L$ around $L_0$. 
In particular the $C^0$-, and the Hausdorff-distance of Lagrangians is not  $d_{\gamma}$-continuous. 
Moreover, it is not hard to see that, given $K$ as above, every closed Lagrangian submanifold in $\mathcal{L}$ can be perturbed in the Hofer distance to a Lagrangian submanifold $L_0$ as in the corollary. It follows that there is a $d_{\hofer}$-open and dense set of $L$ in $\mathcal{L}$ with $\Gamma(\psi;L)>0$.   \end{rem}
\begin{rem}
In \cite{CGG_volume} it was shown that, under some assumptions, the volume of Lagrangian submanifolds is lower semi-continuous with respect to those distances. It would be interesting to know if this is also the case for $\Gamma(\psi;L)$, at least in some situations. 
\end{rem}

The next statement is a corollary of Theorem \ref{thm:barcode_lagr}. For that, the reader should keep in mind that a small $d_{\gamma}$-perturbation of a closed curve in a surface does not necessarily preserve its homotopy class in a complement of a fixed finite set of points. 
\begin{maincor}\label{cor:Gamma_stable2}
Let $(M, \omega)$ be a closed symplectic surface, $\psi:M\to M$ a Hamiltonian diffeomorphism with $h_{\topo}(\psi)>e>0$, and $\mathcal{L}$ a Hamiltonian isotopy class of monotone curves in $M$. Then there exist a finite invariant set~$\mathcal{O}$, a homotopy class $\alpha$ of closed curves in $M\setminus \mathcal{O}$,  and a $d_{\gamma}$-open  (and hence also $d_{\hofer}$-open)  neighbourhood $\mathcal{U}$ of the set of $L$ in $\mathcal{L}$ with $[L] =\alpha$ such that for all $L'\in \mathcal{U}$, 
$$\Gamma(\psi;L')>h_{\topo}(\psi)-e.$$
\end{maincor}

\subsection{Strong barcode entropy and some further examples}\label{sec:intro_strong}
Finally, let me discuss some situations in which positive relative barcode entropy persists on a larger scale. 
We note that the positivity of barcode entropy for the situations  considered so far had its origin in horseshoes, hence in some rather local dynamical features, and the property "$\hbar(\psi;L_0,L_1)>0$"  is not invariant under Hamiltonian isotopies of the Lagrangian submanifolds in general. 
Also note, that while $\hbar(\psi;L_0,L_1)$ is lower semi-continuous 
 with respect to the Hofer- or the $
\gamma$-distance in the pair $(L_0,L_1)$, questions about  stability properties in the diffeomorphism $\psi$ are more subtle. (One result in that direction, but which is obtained rather indirectly, is formulated in Remark~\ref{rem:lsc}.) In order to detect situations with  stronger stability properties, it is convenient to consider a variant of relative barcode entropy. For $R>0$, let us define the  \textit{$R$-strong relative barcode entropy} of  $(\psi,L_0,L_1)$ as  $$\mathsf{H}^R(\psi;L_0,L_1) = \lim_{\substack{\hat{R}\to R\\ \hat{R}<R}}\limsup_{k \to \infty} \frac{\log(b_{\hat{R}k}(\psi;L_0,L_1, k))}{k}.$$
The following properties are either obvious or consequences of stability properties of barcodes with respect to $d_{\gamma}$,  
\cite{Kislev-Shelukhin}, see \S\ref{sec:relbarent}; strict monotonicity in $R$ (point (iii)) was pointed out to me by \c{C}ineli.
\begin{prop}\label{prop:strong_prop} Let $\psi \in \Ham(M,\omega)$, and $(L_0,L_1)$ be an admissible pair. Then the following holds: 
\begin{enumerate}[(i)]
\item If $(L'_0,L'_1)$ is admissible and $L'_i$ is Hamiltonian isotopic to $L_i$, $i=0,1$, then for all $R>0$, $$\mathsf{H}^R(\psi;L'_0,L'_1) = \mathsf{H}^{R}(\psi;L_0,L_1).$$
\item If $\gamma(\psi)\leq R$, then $H^R(\psi;L_0,L_1) = 0$. 
\item $H^R(\psi;L_0,L_1)$, if non-vanishing,  is strictly monotone decreasing in $R$. More precisely: If $\gamma(\psi) >R>{R'}>0$, then $$\mathsf{H}^{{R'}}(\psi;L_0,L_1) \geq \frac{\gamma(\psi) - {R'}}{\gamma(\psi)-R}\mathsf{H}^R(\psi;L_0,L_1);$$ in particular, $$\hbar(\psi;L_0,L_1)\geq \frac{\gamma(\psi)}{\gamma(\psi)-R}\mathsf{H}^R(\psi;L_0,L_1).$$ 
    \item If $d_{\gamma}(\psi,\psi')< \delta$, then for all $R>\delta$, 
    $$\mathsf{H}^{R-\delta}(\psi';L_0,L_1) \geq \mathsf{H}^R(\psi;L_0,L_1).$$
        \item $\mathsf{H}^{nR}(\psi^{n};L_0,L_1) \geq n\mathsf{H}^R(\psi;L_0,L_1)$, for all $n\in \N$.
    \end{enumerate} 
\end{prop}
Hence, given a Hamiltonian isotopy class of Lagrangian submanifolds $\mathcal{L}$, the strong barcode entropies $\mathsf{H}^R(\psi;\mathcal{L}):=\mathsf{H}^R(\psi;L,L)$, $L \subset \mathcal{L}$, are well defined, and the relative barcode entropy is positive independent of the pair of Lagrangians in $\mathcal{L}$ if  $\mathsf{H}^R(\psi;\mathcal{L})>0$. A necessary condition for the latter is that $\Ham(M,\omega)$ has infinite $d_{\gamma}$-diameter, which for example does not hold for the standard $\CP^n$,  \cite{Entov-Polterovich_Calabi}. 
To exhibit first examples with   $\mathsf{H}^R(\psi;\mathcal{L})>0$, 
we consider the twisted products $(M, \omega)=(\Sigma \times \Sigma,\sigma \oplus -\sigma)$, where $(\Sigma,\sigma)$ is a closed symplectic surface with genus at least $2$. 
That the $d_{\gamma}$-diameter of $\Ham(M,\omega)$ is infinite follows from \cite{Usher_boundary_depth} and \cite{Kislev-Shelukhin}.\footnote{More generally, the $d_{\gamma}$-diameter of $\Ham(M,\omega)$ is infinite whenever $(\Sigma, \sigma)$ is a  symplectic aspherical manifold, as was recently established by Mailhot in \cite{mailhot}, and that result provided some motivation for the choice of examples here.}  
\setcounter{mainthm}{5}
\begin{mainthm}\label{mainthm:strong}
Let $I_1, I_2:\Sigma \to \Sigma$ be two commuting anti-symplectic involutions ($I_j^*\sigma = -\sigma$, $I^2_j=\id$, $j=1,2$) such that the complement of $\Fix(I_1)$ consists of two (homeomorphic) components,  with 
 positive genus,  and such that the fixed point set $\Fix(f)$ of  $f:=I_2 \circ I_1$ is non-empty and finite. Let $\mathcal{L}$ be the Hamiltonian isotopy class of the Lagrangian submanifold $L:= \{(x,f(x))\} \subset M$. 
 Then there exists $\psi\in \Ham(M,\omega)$  such that for some $R,E>0$, \begin{align}\label{RE}
     \mathsf{H}^{R}(\psi;\mathcal{L}) \geq E.
     \end{align}
\end{mainthm}
A pair of involutions as in the theorem exist for any closed symplectic surface considered above. Note also that with Proposition \ref{prop:strong_prop} (v), by passing to iterates of $\psi$, one finds, for all positive $R$ and $E$, examples where \eqref{RE} holds.   
As a direct consequence of Proposition~\ref{prop:strong_prop} (iv) and of \cite[Thm.~A]{CGG_barcode}\footnote{In our situation, we will know more about the 
 "homotopy classes" of intersections that generate barcode entropy. As a consequence, lower bounds on $h_{\topo}$ can be obtained in a more elementary way, e.g.\ without the use of Yomdin's theorem, cf.\ \cite{Alves-thesis}.}  we obtain. 
\setcounter{maincor}{6}
\begin{maincor}\label{cor:strong_stable}
Let $(M, \omega)$ be as above. For all $R,E>0$ the subset 
$\{h_{\topo} \geq E\}$ in $\Ham(M, \omega)$  contains open balls of radius $\geq R$ with respect to $d_{\gamma}$. In particular, the same holds with respect to $d_{\hofer}$. 
\end{maincor}
It was previously shown by Chor and the author,  \cite{ChorMeiwes},  that for a closed symplectic surface $(\Sigma, \sigma)$ of genus at least $2$, the assertion of Corollary~\ref{cor:strong_stable} (for the Hofer-distance) holds for  $\Ham(\Sigma,\sigma)$, using rather different methods. Interestingly enough,  and to the best of my knowledge, the following question remains open.  
\begin{qu*}
Do there exist a Hamiltonian isotopy class $\mathcal{L}$ of closed curves in a surface $(\Sigma,\sigma)$ and a Hamiltonian diffeomorphism $\psi$ with  
$\mathsf{H}^R(\psi;\mathcal{L})>0$ for some $R>0$?  
\end{qu*}

\textbf{Idea of proofs.}
As Theorem \ref{mainthm:higherdim} is a relative analogue of Theorem B of  \cite{CGG_barcode}, also the proof is very much a relative analogue of the proof in \cite{CGG_barcode}. In a nutshell, Theorem \ref{mainthm:higherdim} can be established by an interplay of arguments from hyperbolic dynamics together with $C^1$-bounds for pseudo-holomorphic curves with Lagrangian boundary conditions: by results going back to Bowen, the number of orbit segments  from local unstable to local stable manifolds in a locally maximal, topological transitive hyperbolic set $K$ is bounded from below by its topological entropy. Moreover, a combination of $C^1$-estimates for Floer strips and the shadowing theorem shows that 
the energy that is needed to "connect" two chords in $K$ or to "leave" an isolating neighbourhood is uniformly bounded from below. Together, the estimates of Theorem~\ref{mainthm:higherdim} will be obtained.  

The proof of Theorem \ref{thm:barcode_lagr} follows a similar strategy with two additional inputs. First, we apply the fact, recently proved in \cite{Meiwes_horseshoes}, that there exist a specific collection of horseshoe orbits $\mathcal{O}$ and a pair of curves in the complement such that the  growth of chords of the suspension flow connecting those curves and which are unique in their homotopy class in the complement of the link $\mathcal{L}_{\mathcal{O}}$ induced by $\mathcal{O}$, is exponential; moreover, all these chords can be encoded with the help of specific Markov rectangles.  
And second, we use that the Lagrangian Floer homology can be filtered in homotopy classes of chords in the complement of $\mathcal{L}_{\mathcal{O}}$. Together with the estimates already used in the proof of Theorem \ref{mainthm:higherdim}, a "crossing energy lemma" is proved for certain Floer strips crossing $\mathcal{L}_{\mathcal{O}}$, where we have flexibility in the choice of the pair of Lagrangian submanifolds while keeping the diffeomorphism fixed. 

Finally, Theorem \ref{mainthm:strong} is obtained by considering a diffeomorphism $\psi$ supported in the neighbourhood  of the unit sphere bundle of $\Sigma$ embedded in a neighbourhood of the diagonal in $M=\Sigma \times \Sigma$. The support of $\psi$ is foliated by leaves diffeomorphic to that sphere bundle and $\psi$ restricted to a leaf will coincide with the time $T$ map of the geodesic flow for a hyperbolic metric, where $T$ varies between leaves, monotonically increasing first, and decreasing later. The non-trivial chords relative $L$ will come in pairs  with   action difference growing linearly with iterations of $\psi$. The special topological situation then prevents the existence of Floer strips connecting chords between different pairs, as long as one chord belongs to a special subset of pairs. The number of pairs in that subset grows exponentially, which establishes the result.

\textbf{Plan of the paper.}
We start in Section~\ref{sec:preliminaries} with some preliminaries on Lagrangian Floer homology and relative barcode entropy. Section~\ref{sec:Theorem1} contains a lemma about Floer strips with small energy and the proof of Theorem~\ref{mainthm:higherdim}. In Section \ref{sec:Theorem3}, special properties of Lagrangian Floer homology and horseshoes in dimension~2 are discussed, before a proof of Theorem \ref{thm:barcode_lagr} is given. In Section~\ref{sec:strong}, we then construct the examples that establish Theorem~\ref{mainthm:strong}.

\textbf{Acknowledgements.} The author is grateful to Marcelo Alves, Jo\'e Brendel, Erman \c{C}ineli, Viktor Ginzburg, Ba\c{s}ak G\"urel, and Leonid Polterovich for useful discussions. Special thanks go to Urs Fuchs for helpful comments concerning  properties of  Floer strips with small energy.

\section{Preliminaries on relative barcode entropy}\label{sec:preliminaries}
We start with some preliminaries on Hamiltonian  diffeomorphisms and Lagrangian Floer homology. We will work essentially in the same setting as that in \cite{CGG_barcode}. 

\subsection{Hamiltonian diffeomorphisms and mapping cylinder}\label{sec:mappingcyl}
Let $(M,\omega)$ be a symplectic manifold. To keep the exposition simple, we assume that $M$ is closed, but as in \cite{CGG_barcode}, the discussion extends to open $M$ if some assumptions are made on the ends of $(M,\omega)$ and if we consider compactly supported Hamiltonians.  
Let 
  $\psi:M \to M$ be a Hamiltonian diffeomorphism. This means that $\psi = \phi_H^1$ is the time-$1$ map of a (time-dependent) flow $\phi_H^t$ of a $1$-periodic Hamiltonian $H: \R/\Z \times M$, i.e.\ $\frac{d}{dt} \phi^t_H(x) = X_{H_t}(\phi^t_H(x))$ ($(t,x) \in \R \times M)$, where $\omega(X_{H_t}(x),\cdot) = -dH_t(x)$, and $H_t(x) := H(x,t)$. Let 
 $Y:=S^1\times M$, where we identify $S^1 \cong \R /\Z$. 
Fix once and for all a Hamiltonian function $H$ such that $\phi_H^1= \psi$. This defines a lift to the suspension flow $\varphi: \R \times Y \to Y$, defined by 
$$(t,([r],x)) \mapsto  ([t+r],\phi_H^{t+r-\lfloor t+r\rfloor}\circ\phi_H^{\lfloor t + r\rfloor}\circ\phi_H^{-r}(x)),$$
\color{black}
the flow of the vector field $X = \partial_t \oplus X_{H_t}$ on $Y$. For an orbit segment $\eta:[0,k] \to M$, $n\in \N$, of the Hamiltonian flow $\phi_H^t$, we obtain an orbit segment $\widehat{\eta}:[0,k] \to Y$, $\widehat{\eta}(t) := (t,\eta(t))$ of $\varphi$.  

If not explicitly stated otherwise, we will consider a fixed  metric $g$ on $M$ and denote by $d$ resp.\ $\overline{d}$ the distance function on $M$ resp.\ $Y=S^1 \times M$ induced by $g$ resp.\ the product metric $g_{\mathrm{Eucl}} \times g$.

\subsection{The chain complex} 
We briefly recall relevant notions for Lagrangian Floer homology. The construction of Lagrangian Floer homology goes back to Floer \cite{Floer_Lagrangian} and in the monotone setting to Oh \cite{Oh}, the setting that is relevant here. 
More specifically, we closely follow setup and notation of \cite{Usher_boundary_depth} and \cite{CGG_barcode}.   
In the following, all Lagrangian submanifolds $L$ will be closed and monotone. This means that  there exists $\kappa>0$ such that $\langle \omega, A\rangle = \kappa \langle \mu_L, A\rangle$ for all $A\in \pi_2(M,L)$, where $\mu_L$ is the Maslov class of $L$. We also require that $N_L\geq 2$, where the minimal Chern number $N_L$ is defined as the positive generator of the subgroup in $\Z$ generated by the $\langle \mu_L,A\rangle$ for all $A\in \pi_2(M,L)$ that can be represented by cylinders with boundary in $L$, and $N_L=\infty$ if this subgroup is trivial,  see also \cite{Usher_boundary_depth, CGG_barcode}.
 We also denote  by $\mathfrak{h}(M,L) = \kappa N_L \in (0,+\infty]$ the minimal area of a non-trivial cylinder with boundary in $L$. 

Let $L_0 \subset M$ be a Lagrangian submanifold as above, and let  $L_1= \phi^1_G(L_0)$ be Hamiltonian isotopic to $L_0$, for some Hamiltonian $G$ that we fix from now on. Assume that the pair $(L_0,L_1)$ is non-degenerate; we recall that this means that $L_0$ and $\psi^{-k}(L_1)$ intersect transversally, for all $k\in \N$. 
For given $k\in \N$, the construction of the Floer chain complexes  $\mathcal{C}(L_0,L_1,k) = \CF(L_0,\psi^{-k}(L_1))$ can be sketched as follows. 
Let $\mathcal{P}^k = \mathcal{P}(L_0, \psi^{-k}(L_1))$ be the space of smooth paths $x:[0,1] \to M$ with $x(0) \in L_0$, $x(1) \in \psi^{-k}(L_1)$. Fix a system $\gamma_\mathfrak{c}\in \mathcal{P}^k$ of representatives of the homotopy classes $\mathfrak{c}\in  \pi_1(M;L_0,\psi^{-k}(L_1))$. 
 Denote by $\widetilde{\mathcal{P}}^k$ the set of pairs $(x,[v_x])$, where $x\in \mathcal{P}^k$ and $[v_x]$ is a capping up to a certain equivalence relation. Here a \textit{capping} $v_x$ is a $C^1$-map  $v_x:[0,1]^2 \to M$ with $v(\cdot,0) \in L_0$, $v(\cdot,1) \in \psi^{-k}(L_1)$ and $v(0,t) = \gamma_{\mathfrak{c}}$ and $v(1,t) = x(t)$. Two cappings are equivalent if their $\omega$-area and Maslov index coincide;  the Maslov index is well-defined given a choice of symplectic trivializations of  $\gamma_{\mathfrak{c}}^*TM$, see \cite{Usher_boundary_depth}.  Given $\hat{x} = (x,[u_x]) \in \widetilde{\mathcal{P}^k}$, we can glue annuli $v:S^1\times [0,1] \to M$,  with $v(s,0) \in L_0$, $v(s,1) \in \psi^{-k}(L_1)$, to $v_x$ and obtain \textit{recappings} $v\#\hat{x}= (x,[v\#v_x ])$. 
We consider $\mathcal{A}:\widetilde{\mathcal{P}}^k \to \R$, defined by \begin{align}\label{action}\mathcal{A}(x,[v_x]) = -\int_{[0,1]^2} v_x^*\omega.
\end{align}
The critical points of $\mathcal{A}$ are those $(x,[v_x])$ for which $x\in L_0\cap \psi^{-k}(L_1)$ is constant.  We define $\CF(L_0,\psi^{-k}(L_1))$ as in   \cite{CGG_barcode}: fix for each $x\in L_0\cap \psi^{-k}(L_1)$ one equivalence class of cappings $[v_x]$ of $x$. Denote the set of those pairs by $\mathcal{X}^k \subset \widetilde{\mathcal{P}}^k$, and let  
$\CF(L_0,\psi^{-k}(L_1))$ be the vector space  generated by the elements in $\mathcal{X}^k$  over the universal Novikov field $\Lambda$, which is  formed by the formal sums $\lambda = \sum_{j\geq 0} f_jT^{a_j}$, $f_j \in \Z_2$, $a_j\in \R$,  subject to the condition $\#\{j\, |\, f_j \neq 0, a_j\leq C\} < \infty$ for all $C\in \R$. 
To define the boundary map $\partial:  \CF(L_0,\psi^{-k}(L_1)) \to \CF(L_0,\psi^{-k}(L_1))$, one  considers, for a choice of  $S^1$-family $J= (J_t)_{t\in S^1}$ of almost complex structures compatible with $\omega$ and critical points $\hat{x}_{\pm} = (x_{\pm},[v_{{\pm}}])$ of $\mathcal{A}$, the moduli space $\mathcal{M}_{J}(\hat{x}_-, \hat{x}_+,L_0,\psi^{-k}(L_1))$ of  holomorphic strips  $u:\R \times [0,1] \to M,$
\begin{align}\label{hol_eq}
\partial_s u(s,t)+ J_t(u(s,t))\partial_t u(s,t) = 0,
\end{align}
subject to boundary conditions $u(s,0) \subset L_0$, $u(s,1) \in \psi^{-k}(L_1)$ ($s\in \R$),    
asymptotics   $\lim_{s\to \pm\infty}u(s,t) \equiv x_{\pm}$, and for which  the glued map $v_{-} \# u$ is equivalent to $v_{+}$. 
If $J$ is such that  the moduli spaces are transversally cut out, which holds generically,  one can show that if the Maslov index difference of $\hat{x}$  and $v\#\hat{x}_j$ is $1$, $\mathcal{M}_{J}(\hat{x}_-, \hat{x}_+,L_0,\psi^{-k}(L_1))/\R$ is finite, where we are modding out by the canonical $\R$-action on solutions of \eqref{hol_eq}. This makes it possible to define the boundary map $\partial$, defined on the generators $\mathcal{X}^k$ as 
 $\partial \hat{x}  = \sum \lambda_{j} \hat{x}_j$, where 
 $\lambda_j = \sum_{v} f_v T^{\omega(v)}\in \Lambda$.  Here the sum is taken over all recappings $v$ of $\hat{x}_j$   such that the Maslov index difference of $\hat{x}$  and $v\#\hat{x}_j$ is $1$,
 and  $f_v$ is the number of elements (mod $2$) in $\mathcal{M}_J(\hat{x},v\#\hat{x}_j, L_0,\psi^{-k}(L_1))/\R$. 
    One shows that $\partial$ is indeed a boundary map $(\partial^2 = 0)$ and that the homology denoted by $\HF(L_0,\psi^{-k}(L_1))$ does not depend on the choice of $J$. 

The homology is \textit{invariant} under Hamiltonian isotopy: a Hamiltonian isotopy  $L^s_0= \phi_{F}^s(L_0)$, $F:M\times \R/\Z \to \R$,  $s\in [0,1]$, from $L_0$ to $L'_0:= L^1_0$ with $(L'_0, L_1)$  non-degenerate, induces for every $k\in \N$ a chain map $f_0:\CF(L_0,\psi^{-k}(L_1)) \to \CF(L'_0,\psi^{-k}(L_1))$, defined by  counting holomorphic strips $u:\R\times S^1 \to M$ (with $s$-dependent $J$) in a similar manner as above, now with \textit{moving Lagrangian boundary conditions} 
$u(s,0) \in L^{\beta(s)}_0$ and $u(s,1) \in \psi^{-k}(L_1)$ $(s\in \R)$, where 
$\beta:\R \to [0,1]$ is a smooth non-decreasing surjective function, constant outside a compact interval \cite[\S 5]{Oh}, see also \cite[\S 3.2]{Biran_Cornea13}. Similarly, one defines $g_0:\CF(L'_0,\psi^{-k}(L_1)) \to \CF(L_0,\psi^{-k}(L_1))$
 and one obtains that $g_0 \circ f_0$ is chain homotopic to the identity. Analogously, for a Hamiltonian isotopy $L_1^s$, $s\in [0,1]$, from $L_1$ to $L'_1$, one obtains for all $k\in \N$ a chain map 
  $f_1:  \CF(L_0,\psi^{-k}(L_1)) \to \CF(L_0,\psi^{-k}(L'_1))$ with similar properties.   

\subsection{Rescaling and lifts of holomorphic strips}
Given a holomorphic strip $u\in\mathcal{M}_{J}(\hat{x}_-,\hat{x}_+,L_0,\psi^{-k}(L_1))$, we will consider  the \textit{rescaled strip} $\check{u}:\R \times [0,k] \to M$ defined by $\check{u}(s,t) = u(\frac{1}{k} s, \frac{1}{k}t)$. The strip $\check{u}$ is holomorphic, that is, it satisfies \eqref{hol_eq} with $J_t$ replaced by  $J_{\frac{1}{k}t}$. 
By composing with the Hamiltonian flow, we obtain   
$$\hat{u}:\R\times [0,k] \to M, \, \, \hat{u}(s,t) := \phi^t_H(\check{u}(s,t)).$$ The strip~$\hat{u}$ satisfies the Floer equation 
\begin{equation}\label{Floer_eq}
\partial_s \hat{u}(s,t) + \hat{J}_{t}(\hat{u}(s,t))(\partial_t \hat{u}(s,t) - X_{H_t}(\hat{u}(s,t))) = 0,
\end{equation}
where $\hat{J}_t := (\phi_H^t)_{*} J_{\frac{1}{k}t}$, $t\in [0,k]$.
The asymptotics of $\hat{u}$ are chords from $L_0$ to $L_1$ ($\lim_{s\to\pm \infty}\hat{u}(s,t) = \phi^{t}_{H}(x_{\pm})$, $t\in [0,k]$)  and $\hat{u}$ satisfies $\hat{u}(s,0)\in L_0$, $\hat{u}(s,k) \in L_1$, $s\in \R$. This identification extends in a natural way to the situation of moving boundary conditions and $s$-dependent $J$.

Additionally, we consider also the maps
\begin{align*}
&\overline{u}: \R \times [0,k]\to \R/\Z \times M = Y, \,\,  \overline{u}(s,t):= ([t], \hat{u}(s,t)), \\
&\widetilde{u}: \R \times [0,k] \to \R \times \R/\Z \times M,\,  \, \widetilde{u}(s,t) = (s,[t], \hat{u}(s,t)), \\
&\dtilde{u}:\R \times [0,k] \to \R \times [0,k]\times M, \, \, \dtilde{u}(s,t) := (s,t,\hat{u}(s,t)).
\end{align*} 
With respect to  a suitable  (time-independent)  $\widetilde{J}$, the lift $\dtilde{u}$ is  $\widetilde{J}$-holomorphic. 
Here, $\widetilde{J}$ is defined at a point  $(s,t,p) \in \R \times [0,k] \times M$ by $\widetilde{J}_{(s,t,p)}(\partial_s) =\partial_t + X_H(s,t,p)$, where $X_H(s,t,p)$ is the (time-independent) horizontal lift  of $X_{H_t}(p)$ to $\R \times [0,k] \times M$, and by the condition that 
$\widetilde{J}_{(s,t,p)}(v)$ is the horizontal lift of $\hat{J}_{t}(v)$ to $\{s\} \times \{t\} \times M$. 
The boundary components of $\dtilde{u}$ then lie in $\R \times \{0\} \times L_0$ and $\R \times \{k\} \times L_0$.

In this situation, we can also naturally generalize to moving boundary conditions and $s$-dependent $J$. In particular, the condition with moving boundary $L^s_i$, $i=0,1$, corresponds to the condition that the boundary components of $\dtilde{u}$ lie in $\{(s,0,p)\, |\, s\in \R, p \in L_0^{\beta(\frac{1}{k}s)}\}$ and $\{(s,k,p)\, |\, s\in \R, p \in L_1^{\beta(\frac{1}{k}s)}\}$, where $\beta$ is defined as above.
\subsection{Relative barcode entropy}\label{sec:relbarent}
We first keep the assumption that the pair $(L_0,L_1)$ is non-degenerate. The action $\mathcal{A}$ defined in \eqref{action} can be extended to general elements of $\CF(L_0,\psi^{-k}(L_1))$. We set $\nu(\lambda) := \min a_j$  for $\lambda = \sum_{j\geq 0} f_jT^{a_j}\in \Lambda$,  $\nu(0) = + \infty$, and put $$\mathcal{A}(\lambda \hat{x}) := \mathcal{A}(\hat{x}) -\nu(\lambda)  
\text{ and } \mathcal{A}\left(\sum \lambda_i\hat{x}_i\right):= \max_i \mathcal{A}(\lambda_i \hat{x}_i).$$

The differential $\partial$ strictly decreases the action $\mathcal{A}$. In fact, for any $u\in\mathcal{M}_{J}(\hat{x}_-,\hat{x}_+,L_0,\psi^{-k}(L_1))$,
$$\mathcal{A}(x_-) - \mathcal{A}(x_+) = E(u).$$ 
Here, $E(u)$ is the energy of $u$ given by  
$$E(u) = \frac{1}{2}\int_{-\infty}^\infty \|du\|^2 =  \int_{-\infty}^\infty \int_{0}^1 |\partial_s u(s,t)|^2 \, dt  \, ds,$$
where $|\cdot|$ is induced by the metric $\omega(\cdot, J\cdot)$.

We are now in the situation to give the definition of (relative) barcode entropy. We keep the discussion rather short, making use of the results in \cite{Usher_Zhang} and refer the reader to \cite{CGG_barcode} and references therein for an equivalent definition and further motivation.  Given $\epsilon>0$, a non-zero vector $\zeta \in \im \partial \subset \CF(L_0,\psi^{-k}(L_1))$ is said to be \textit{$\epsilon$-robust} if for every $y\in \CF(L_0,\psi^{-k}(L_1))$ with $\partial y = \zeta$, 
$$\mathcal{A}(y) - \mathcal{A}(\zeta)>\epsilon.$$  
 A subspace $V \subset \im \partial \subset \CF(L_0,\psi^{-k}(L_1))$ is called \textit{$\epsilon$-robust} if every non-zero vector $\zeta \in V$ is $\epsilon$-robust.  
Define  $$b^*_{\epsilon}(\psi;L_0,L_1,k) := \max\{ \dim V \, |\, V\subset\CF(L_0,\psi^{-k}(L_1)) \text{ is } \epsilon\text{ - robust}\}.$$ 
By results in \cite{Usher_Zhang} and the discussion in \cite{CGG_barcode}, the integer $b^*_{\epsilon}(\psi;L_0,L_1,k)$ coincides with the number of finite bars of length larger than  $\epsilon$ in the definition of the barcode of $\CF(L_0,\psi^{-k}(L_1))$ in \cite{CGG_barcode}. 
To keep the same conventions as in \cite{CGG_barcode}, we also consider  the infinite bars and set $$b_{\epsilon}(\psi;L_0,L_1,k) = b^*_{\epsilon}(\psi;L_0,L_1,k) + \dim \HF(L_0, \psi^{-k}(L_1)).$$
For a general admissible, not necessarily non-degenerate pair $(L_0,L_1)$, define   
$$b_{\epsilon}(\psi;L_0,L_1,k) = \liminf_{(L'_0,L'_1)\to (L_0,L_1)} b_{\epsilon}(\psi;L'_0,L'_1,k),$$ where the limit is taken over all non-degenerate pairs $(L'_0,L'_1)$ and the convergence of pairs is considered with respect to the $C^{2}$ topology. By stability properties of barcodes, see below, the above is well defined.

Following \cite{CGG_barcode}, define the \textit{$\epsilon$-relative barcode entropy} as $$\hbar_{\epsilon}(\psi;L_0,L_1) = \limsup_{k\to \infty} \frac{\log(b_{\epsilon}(\psi;L_0,L_1,k))}{k},$$ and the \textit{relative barcode entropy} as 
$$\hbar(\psi;L_0,L_1) = \lim_{\epsilon\to 0} \hbar_{\epsilon} (\psi;L_0,L_1).$$
Additionally, for any $R>0$, let us also recall from the introduction the definition of the \textit{$R$-strong relative barcode entropy}, 
$$\mathsf{H}^R(\psi;L_0,L_1) = \lim_{\substack{\hat{R}\to R\\ \hat{R}<R}}\limsup_{k\to \infty} \frac{\log (b_{\hat{R}k}(\psi;L_0,L_1,k))}{k}.$$
Some important properties of the (strong) barcode entropy follow from known stability properties of the filtered Floer homology with respect to the Hofer-norm and the $\gamma$-norm.
Recall that for a  Hamiltonian diffeomorphism $\psi$ on $(M, \omega)$, its \textit{Hofer}-norm is defined as 
$$\|\psi\|_{\hofer} = \inf_{H} \int_{0}^1 \max H(x,t) - \min H(x,t) \, dt,$$  and its $\gamma$-norm as 
$$\gamma(\psi) = \inf (c(H) + c(\overline{H})). $$
In both cases the  infimum is taken over all (time-dependent) Hamiltonian functions $H$ such that $\psi = \phi_H^1$, moreover $c$ denotes the spectral number associated to the fundamental class $[M]$, and $\overline{H}$ denotes the Hamiltonian with $\phi_{\overline{H}}^t = (\phi_H^t)^{-1}$, see \cite{Hofer90}, and \cite{Viterbo92,Schwarz2000, Oh05}, respectively.  It is well-known that $\gamma(\psi) \leq \|\psi\|_{\hofer}$.  
Given a Hamiltonian isotopy class $\mathcal{L}$ of closed monotone Lagrangian submanifolds, the \textit{exterior $\gamma$-distance} on $\mathcal{L}$ is defined as $$d_{\gamma}(L,L') : = \inf \{\gamma(\psi)\, |\, \psi \in \Ham(M,\omega), \psi(L) = L'\}.$$
Analogously the Hofer-distance $d_{\hofer}$ is defined.   

For our purposes it is convenient to express the stability properties in terms of the quantities  $b_{\epsilon}(\psi;L_0,L_1,k)$ for triples $(\psi,L_0,L_1)$ as above and $k\in \N$: for any $\epsilon>0$, it holds that if $L'_0$ is a Lagrangian in $M$ with $d_{\gamma}(L'_0,L_0)< \delta<\epsilon$ and such that $(L'_0,L_1)$ is admissible, then 
\begin{align}\label{b_epsilon1}
b_{\epsilon-\delta}(\psi;L'_0,L_1,k) \geq b_{\epsilon}(\psi;L_0,L_1,k).
\end{align}
This statement follows from the work \cite{Kislev-Shelukhin}; the weaker statement that the same holds with respect to Hofer-distance already follows from a straightforward adaption of results in \cite{polterovich2016autonomous}.
The same results imply the following statement: if $\psi'$ is a Hamiltonian diffeomorphism on $M$  with $d_{\gamma}(\psi',\psi):=\gamma(\psi^{-1} \circ \psi')< \delta < \epsilon$, then
\begin{align}\label{b_epsilon2}
b_{\epsilon-\delta}(\psi';L_0,L_1,1) \geq b_{\epsilon}(\psi;L_0,L_1,1),
\end{align}
and hence, by the triangle inequality, 
\begin{align*}
    b_{\epsilon-k\delta}(\psi';L_0,L_1,k) \geq b_{\epsilon}(\psi;L_0,L_1,k), 
    \end{align*}
    if $k\delta < \epsilon$. 
By \eqref{b_epsilon1} and the fact that 
\begin{align}\label{symmetric}
b_{\epsilon}(\psi;L_0,L_1,k) = b_{\epsilon}(\psi^{-1};L_1,L_0,k),
\end{align}
see \cite{CGG_barcode}, it follows that 
$\hbar(\psi;L_0,L_1)$ is lower semi-continuous in $(L_0,L_1)$ with respect to $d_{\gamma}$. Furthermore,  Proposition \ref{prop:strong_prop} from the introduction can be proved using the facts above.
\begin{proof}[Proof of Proposition  \ref{prop:strong_prop}]
If $(L'_0,L'_1)$ are Hamiltonian isotopic to $(L_0,L_1)$, then by \eqref{b_epsilon1} and \eqref{symmetric}, for $R>0$ and $k$ sufficiently large, $$b_{Rk-Z}(\psi; L'_0,L'_1,k)\geq b_{Rk}(\psi;L_0,L_1,k),$$ where $Z:=d_{\gamma}(L_0',L_0) +d_{\gamma}(L'_1,L_1)$. Assertion (i) follows directly  if we apply this inequality for any given $\epsilon>0$  to $\hat{R} = R-\epsilon$: for $k>Z/\epsilon$,  
$$b_{(R-2\epsilon)k}(\psi; L'_0,L'_1,k)\geq b_{\hat{R}k-Z}(\psi; L'_0,L'_1,k)\geq b_{\hat{R}k}(\psi;L_0,L_1,k).$$   
To obtain (iii), it is sufficient to show that for any $0<{R'}<R< \gamma(\psi)=:Z$ and 
 any sequence $k_i$ of natural numbers with $k_i \to +\infty$ there exists (up to passing first to a subsequence) a sequence $k'_i$ of natural numbers such that 
$$\lim_{i \to +\infty} \frac{\log b_{R'k'_i-Z}(\psi;L_0,L_1, k'_i)}{k'_i} \geq \frac{\gamma(\psi) - R'}{\gamma(\psi) - {R}}  \lim_{i \to +\infty} \frac{\log b_{{R}k_i}(\psi;L_0,L_1, k_i)}{{k}_i}.$$
Let us choose   $k'_i := \lceil t'_i \rceil$ with $t'_i := \frac{\gamma(\psi) - R}{\gamma(\psi) - {R'}} k_i$. 
Since we have  $d_{\gamma}(\psi^{k'_i}, \psi^{k_i}) \leq (k_i-k'_i)\gamma(\psi)$, \eqref{b_epsilon2} implies that  $$b_{Rk_i -\gamma(\psi)(k_i-k'_i)}(\psi;L_0,L_1, k'_i)\geq b_{Rk_i}(\psi; L_0,L_1, k_i).$$
Note that $$Rk_i -\gamma(\psi)(k_i-k'_i) \geq  k_i(R-\gamma(\psi)) + \gamma(\psi)t'_i = R't'_i.$$ 
Hence,  $$ b_{R't'_i}(\psi;L_0,L_1, k'_i)\geq  b_{Rk_i}(\psi; L_0,L_1, k_i). $$
The left-hand side is at most $b_{R'k'_i-Z}(\psi;L_0,L_1, k'_i)$, and (iii) follows. 
That $\gamma(\psi) = R$ implies $\mathsf{H}^R(\psi;L_0,L_1)=0$ follows easily from the inequality in (iii), and hence also assertion (ii) directly follows. 
Assertion (iv) can be easily shown using 
 \eqref{b_epsilon2}, and  (v) is obvious. 
\end{proof}
For further properties of barcode entropy, we refer to \cite[Prop.~4.4]{CGG_barcode}.

\section{Pseudo-orbits 
 property for small energies and the proof of Theorem \ref{mainthm:higherdim}}\label{sec:Theorem1}
\subsection{Holomorphic curves with small energy and pseudo-orbits}
We start this section with a lemma that asserts that if a holomorphic strip $u$ has sufficiently small energy, then the paths $t\mapsto\overline{u}(s,t)$ are $\varphi$-pseudo-orbits in $Y$ along the intervals where they are defined. See also Lemma~6.3 in \cite{CGG_barcode} for a related statement.   
Let $\hat{J}$ be a fixed (possibly $t$-dependent)   almost complex structure on $M$, and let \begin{align}\label{specialJ}
J = J_t =  [D\phi_H^t]^{-1}\hat{J}D\phi_H^{t}.
\end{align}
\begin{lem}\label{lem:dubound}
Let $(L_0,L_1)$ be a non-degenerate admissible pair of Lagrangian submanifolds in $M$. Then, given $\eta>0$, there exists $\epsilon= \epsilon(\eta,L_0,L_1)>0$ such that if $u\in \mathcal{M}_{J_{kt}}(\hat{x}_-,\hat{x}_+,L_0,\psi^{-k}(L_1))$ (for some $\hat{x}_\pm$ and $k\in \N$) 
and $E(u)<\epsilon$, then
for all $s\in \R$ and all $\hat{n} \in \Z$, $0\leq \hat{n} \leq k-1$, $t \in [0,1]$,  \begin{align}\label{eq:pseudo}
\overline{d}(\overline{u}(s,t+ \hat{n}),\varphi^{t}(\overline{u}(s,\hat{n}))) < \eta.
\end{align}   
Moreover, given a $C^2$-small neighbourhood of any fixed   admissible pair $(L'_0,L'_1)$, then there exists  $\epsilon=\epsilon(\eta)>0$ such that the above holds for all non-degenerate admissible pairs $(L_0,L_1)$ in that neighbourhood.  
\end{lem}
\begin{proof}
For $u\in \mathcal{M}_{J_{kt}}(\hat{x}_-,\hat{x}_+,L_0,\psi^{-k}(L_1))$ as in the lemma, and $\hat{n} \in \Z$, $0\leq \hat{n} \leq k-1$,  consider $v:\R \times [-\hat{n},k-\hat{n}] \to M$, defined by 
$$v(s,t) := \psi^{\hat{n}}(\check{u}(s,t+\hat{n}))= \psi^{\hat{n}}\left({u}\left(\frac{1}{k}s,\frac{1}{k}(t+\hat{n})\right)\right).$$ 
We note that $v$ is $J_t$-holomorphic and that $v|_{\R \times \{-\hat{n}\}} \subset \psi^{\hat{n}}(L_0)$, $v|_{\R \times \{k-\hat{n}\}} \subset \psi^{\hat{n}-k}(L_1)$. 
Also note that 
\begin{align}\label{u1}
\begin{split}
\overline{u}(s,\hat{n})&=(0,v(s,0)), \\  
\overline{u}(s,t+\hat{n}) &= (t,  \phi_H^{t+\hat{n}}(\check{u}(s,t+\hat{n}))) = \varphi^t((0, v(s,t))),
\end{split}
\end{align} for all $s\in \R, t\in [0,1]$. \color{black}

To simplify the exposition, let us assume that $k\geq 2$, and it will be obvious how to treat the case $k=1$. For $(s,t) \in \R^2$ and $r>0$,  denote by $D_r(s,t)$ the open disk of radius $r$ around $(s,t)$.
We note the following. If $\hat{n}\notin\{0, k-1\}$, then for any point $(s,t) \in \R \times [0,1]$ we have that $D_1(s,t)$ is contained in the domain $\dom(v)$ of $v$; if  $\hat{n} =0$, then for any $s\in \R$  we have that $D_2(s,0) \cap (\R \times [0,+\infty))\subset \dom(v)$; and if $\hat{n} = k-1$, then for any $s\in \R$ we have that  $D_2(s,1)\cap (\R \times (-\infty,1])\subset \dom(v)$. 
We can apply a priori estimates for the differential of $v$ for $(s,t) \in \R \times [0,1]$, see  \ \cite[Lemma~4.3.1]{McDuffSalamon}\footnote{That Lemma is  formulated for a fixed almost complex structure, but it can be easily proved that it continuous to  hold  for $(s,t)$-dependent (uniformly bounded) almost complex structures.}, and 
obtain that there exist $\delta>0, c>0$ (depending on $L_0$ and $L_1$), 
such that if $\int_{\R \times [-\hat{n},k-\hat{n}]} |dv|^2\,  ds \wedge dt< \delta$, then for all $s\in \R, t\in [0,1]$, 
$$|dv(s,t)|^2 \leq c \int_{\hat{D}} |dv|^2 \, ds \wedge dt< c\delta, $$
where
$$\hat{D}:= \begin{cases}&D_1(s,t) \, \, \, \quad\qquad\qquad\qquad \text{ if } \hat{n} \notin\{0, k-1\}, \\ &D_2(s,0) \cap (\R \times [0,+\infty)) \, \text{ if } \hat{n} =0, \\ &D_2(s,1) \cap (\R \times (-\infty, 1]) \, \text{ if }\hat{n} = k-1.\end{cases}$$  

We note that $\int_{\R \times [-\hat{n}, k-\hat{n}]} |dv|^2 \, ds \wedge dt = 2 E(u)$, and hence if $E(u)$ is small, then for any $s \in \R$ the curve $t\mapsto v(s,t)$, $t\in [0,1]$, stays uniformly close to the point  $v(s,0) = \psi^{\hat{n}}(\check{u}(s,\hat{n}))$. When applying the flow $\varphi^t$ to $(0,v(s,0))$ and $(0,v(s,t))$ for $0\leq t\leq 1$, this means with \eqref{u1} that, given $\eta>0$,  there exist $\epsilon>0$ such that if $E(u)< \epsilon$, then \eqref{eq:pseudo} holds for $0\leq t\leq 1$. This shows the first assertion. Moreover,  $\delta$ and $c$ above can be chosen to vary  continuously in $L_0$ and $L_1$ in the $C^2$ topology, cf.\ \cite[Rmk.\ 4.3.2]{McDuffSalamon}, and  are defined even for degenerate pairs $(L_0,L_1)$. The second assertion of the lemma follows. 
\end{proof}
We will also use the following, slightly more general   statement, which can be obtained as a  direct consequence of Lemma \ref{lem:dubound}. 
\begin{cor}\label{cor:pseudo_2}
For all $\eta>0$, $B\in\N$, there exists $\epsilon= \epsilon(\eta,B)>0$ such that if $u\in \mathcal{M}_{J_{kt}}(\hat{x}_-,\hat{x}_+,L_0,\psi^{-k}(L_1))$ (for some $\hat{x}_\pm$ and $k\in \N$) 
and $E(u)<\epsilon$, then for all $s\in \R$ and all $\hat{t} \in [0,k]$, $t \in [-B, B ]$ with $t+\hat{t} \in [0,k]$, 
\begin{align}
\overline{d}(\overline{u}(s,t+ \hat{t}),\varphi^{t}(\overline{u}(s,\hat{t}))) < \eta. 
\end{align} 
In particular, for all $s\in \R$, and all $\hat{n},n \in \Z$, $0\leq \hat{n} \leq k$, $-B \leq n\leq B$, $0\leq n+\hat{n} \leq k$, 
\begin{align}\label{eq:pseudo_discrete}
d(\hat{u}(s,n+ \hat{n}),\psi^n(\hat{u}(s,\hat{n}))) < \eta.
\end{align}
\end{cor}
Note that a \textit{$\eta$-pseudo-orbit} (for $\psi$) is a sequence  $\{y_n\}_{n\in \Z}$ in $M$ for which $d(y_n,\psi(y_{n-1})) < \eta$ for all $n\in \Z$. By Lemma \ref{lem:dubound}, if $E(u)$ is sufficiently small (independent of $k$), then the $\R$-family of orbit segments $u(s,n)$, $0\leq n\leq k$, $s\in \R$, can be completed to a family of $\eta$-pseudo-orbits.

\subsection{Local stable/unstable manifolds and growth of chords}\label{sec:stable}
We recall some facts on local stable and unstable manifolds in hyperbolic sets, see e.g.\   
\cite[\S 6.4]{Hasselblatt-Katok}. Let first, in general, $\psi:M \to M$ be a $C^{\infty}$-diffeomorphism on a closed manifold $M$. 
Let $K$ be a compact hyperbolic set for $\psi$.  
For $q\in K$, we define the local stable resp.\ unstable manifold of $q$ of size $\delta$ as 
$$W^s_{\delta}(q) = \{y \in M \, |\, d(\psi^n(y), \psi^n(q)) \leq \delta \text{ for all } n\geq 0\},$$
$$W^u_{\delta}(q) = \{y \in M \, |\, d(\psi^n(y), \psi^n(q)) \leq \delta \text{ for all } n\leq 0\}.$$
There exists $\delta^*= \delta^*(K)>0$ such that for all $\delta\leq\delta^*$, $q\in K$,  the local stable and unstable manifolds 
 $W^s_{\delta}(q)$ resp.\  $W^u_{\delta}(q)$ are embedded disks,  
 and there exist $\lambda, c>0$ such that 
\begin{align}\label{contraction}
    \psi^n(W^s_{\delta}(q)) \subset W^s_{\delta ce^{-\lambda n}}(\psi^n(q)) \text{ for all } n\geq 0,
\end{align}
\begin{align}\label{expansion}
    \psi^n(W^u_{\delta}(q)) \subset W^u_{\delta ce^{\lambda n}}(\psi^n(q)) \text{ for all } n\leq 0.
\end{align}
Furthermore, $\delta^*$ is an expansivity constant for $K$: for any $x,y\in K$, if for all $k\in 
\Z$, $d(\psi^k(x), \psi^k(y)) \leq \delta^*$, then $x=y$. 
The local stable resp.\  unstable disks at $q\in K$ 
vary continuously in $q\in K$ in the $C^{\infty}$ topology.  
Note that by \eqref{contraction} and \eqref{expansion}, if $(M,\omega)$ is symplectic and $\psi^*\omega = \omega$, we must have that   $\omega|_{W^{s}_{\delta^*}(q)} =0$  and $\omega|_{W^{u}_{\delta^*}(q)}= 0$, i.e.,   ${W^{s}_{\delta^*}(q)}$ and ${W^{u}_{\delta^*}(q)}$ are Lagrangian.

For $x\in M$, we denote by $\mathcal{O}(x) = \{\psi^n(x)\, |\, n\in \Z\}$ the full orbit of $x$ and by $\mathcal{S}^m(x):= \{x,\psi(x), \ldots, \psi^m(x)\}$, $m\geq 0$, the orbit segment of length $m$ starting at $x$. If additionally $\mathcal{S}^m(x)\subset K$, and   $x \in W^u_{\delta}(q)$, $\psi^m(x) \in W^s_{\delta}(p)$ for some $q,p \in K$, then we say that $\mathcal{S}^m(x)$ is \textit{a $(q,p,\delta)$-chord of length~$m$}. 
Note that for $x,q,p\in K$, if $d(\psi^{n}(x), \psi^{n}(q))\leq \delta$ for all $n\leq 0$ and $d(\psi^{n+m}(x),\psi^{n+m}(p))\leq\delta$ for all $n\geq 0$, then  $\mathcal{S}(x)$ is a $(q,p,\delta)$-chord of length $m$.

For $q,p \in K$, $\delta>0$, and $m\in \N$, denote by $N(q,p,\delta, m)$ the number of $(q,p, \delta)$-chords of length at most $m$.
The following is a relative version of the well-known result that $h_{\topo}(\psi|_K)$ equals the exponential growth rate of periodic orbits in $K$, if $K$ is additionally locally maximal (see Thm.~18.5.1 in \cite{Hasselblatt-Katok}). 
\begin{prop}\label{prop:entropy_chord}
If $K$ is a locally maximal, topologically transitive  hyperbolic set, then for any $q,p\in K$, $0 < \delta \leq \delta^* = \delta^*(K)$,
$$\limsup_{m\to \infty}\frac{\log N(q,p,\delta, m)}{m} =h_{\topo}(\psi|_K).$$
\end{prop}
\begin{proof}
The proof is a straightforward adaption of the proof of Thm.~18.5.1 in \cite{Hasselblatt-Katok}. We sketch it for the convenience of the reader. The inequality $\limsup_{m\to \infty}\frac{\log(N(q,p,\delta,m))}{m} \leq h_{\topo}(\psi|_K)$ follows from the fact that $\psi|_K$ is expansive with expansivity constant $\delta^*$. Indeed, for any two distinct  $(q,p,\delta)$-chords $\mathcal{S}^m(x_1)$ and $\mathcal{S}^m(x_2)$  of length $m\geq 1$, there exist $1\leq k\leq m$ such that $d(\psi^k(x_1), \psi^k(x_2))>\delta^*$. Otherwise, $d(\psi^k(x_1), \psi^k(x_2))\leq\delta^*$ for all $k\in \Z$, and hence $x_1= x_2$. 

For the reverse inequality, assume first that $f|_K$ is topologically mixing.  
We recall Bowen's specification theorem \cite[Thm.~2.10]{Bowen-measure}. A \textit{specification} is a pair $s=(\tau,P)$, where $\tau = \{I_1,\ldots, I_l\}$ is a finite collection of disjoint intervals of integers, and $P: I := \bigcup_{i=1}^l I_l \to K$ is map with $\psi^{k_2-k_1}(P(k_1)) = P(k_2)$  whenever $k_1,k_2 \in I_j$ for some $1\leq j\leq l$. A specification $s=(\tau,P)$ is said to be \textit{$n$-delayed} if there is an interval of length at least $n$ between every pair of intervals belonging to $\tau$.
For $\epsilon>0$ and a specification $s=(\tau,P)$, let $$V(s,\epsilon) := \{ x\in K\, |\, d(\psi^k(x),P(k)) < \epsilon \text{ for all } k\in I\}.$$   
Since $f|_K$ is assumed to be topologically mixing, Bowen's specification theorem asserts that  for any $\epsilon>0$, 
there exist $M_{\epsilon}\in \N$ such that $V(s,\epsilon)\neq \emptyset$ for any $M_{\epsilon}$-delayed specification $s$. 
If $s=(\tau,P)$ with $\tau= \{I_1, \ldots, I_l\}$ is such that $I_1 = \{a\}$ and $I_l = \{b\}$ with $P(a) = q$ and $P(b) =p$, then in fact, for any fixed $\delta < \delta^*$ one can choose $M_{\epsilon} = M_{\epsilon}(\delta)$ sufficiently large, such that the orbit $\mathcal{O}(x)$ of $x\in V(s,\epsilon)$ that is obtained in the proof in \cite{Bowen-measure} contains by construction a segment that is a $(q,p,\delta)$-chord of length $b-a$.  It follows that for any point $y$ in a $(n,\epsilon)$-separated set $E\subset K$, there exists a $(q,p,\delta)$-chord $\mathcal{S}(x)$ of length $n+ 2M_{\epsilon/2}$ such that $d(\psi^{M_{\epsilon/2} + k}(x),\psi^{k}(y))< \epsilon/2$ for all $0\leq k \leq n$. Hence there are at  least $\# E$-many $(q,p,\delta)$-chords of length at most $n+ 2M_{\epsilon/2}$. This implies that $\limsup_{m\to \infty}\frac{\log N(q,p,\delta, m)}{m} \geq h_{\topo}(\psi|_K)$ in the topologically mixing case. The generalisation to the  topologically  transitive case follows then from the spectral decomposition theorem, see \cite[Thm.~18.3.1]{Hasselblatt-Katok}.
\end{proof}
\begin{proof}[Proof of Theorem \ref{mainthm:higherdim}]
Let $K$ be a locally maximal, topological transitive, compact  hyperbolic set.
By compactness of $K$ and since $\partial U'\cap K = \emptyset$ for any isolating neighbourhood $U'$ of $K$, we can choose $\delta'>0$ such that 
\begin{enumerate}[(i)]
\item $U = N_{\delta'}(K):=\{y\in M\, |\, \inf_{x\in K}d(y,x)\leq \delta'\}$ is an isolating neighbourhood for $K$,  
\item $\delta' \leq \delta^*/2$, where $\delta^* = \delta^*(K)$ is the constant discussed above.
\end{enumerate}
Let  $(L_0,L_1)$ be an admissible pair with 
$W^u_{\delta}(q) \subset L_0$, and $W^s_{\delta}(p)\subset L_1$, for some  $q, p \in K$, $\delta \in (0,\delta']$. Assume first that $(L_0,L_1)$ is also non-degenerate. 
Let $J$ be defined as in \eqref{specialJ}.  
By Proposition~\ref{prop:entropy_chord}, and by Prop.~3.8 in \cite{CGG_barcode} (cf.\ Lemma \ref{prop:algebaric_barcode}), it is sufficient to show that there exists {$\epsilon>0$} such that if $u\in\mathcal{M}_{J_{kt}}(\hat{x}_-, \hat{x}_+,L_0,\psi^{-k}(L_1))$, $k\in \N$, with  $\hat{x}_-=(x_-,[u_-])$, $\hat{x}_+ = (x_+,[u_+]) \in  \widetilde{\mathcal{P}}_k$, satisfies that 
\begin{itemize}
    \item $\mathcal{S}^k(x_-(0))$ or $\mathcal{S}^k(x_+(0))$ is a  $(q,p,\delta)$-chord; 
    \item $E(u) < \epsilon$,
    \end{itemize}
    then  $u$ is constant. Choose $\eta>0$ such that 
    \begin{enumerate}
\item any $\eta$-pseudo-orbit that intersects $\partial U$ is not entirely contained in $\overline{U}$;
\item any continuous family of $\eta$-pseudo-orbits $\mathcal{Q}_s = \{y^s_k\}_{k\in \Z}$, $s\in \R$, in $U$ is $\delta$-shadowed by a continuous family of orbits $(\mathcal{O}(x^s))_{s\in \R}$ in $K$, i.e., $d(y^s_k,\psi^k(x^s))<\delta$ for all $k\in \Z$, $s\in \R$.
\end{enumerate}
The point (1) can be achieved by compactness of $\partial U$ and since $\partial U \cap K = \emptyset$, the point (2) by the shadowing theorem (see e.g.\ Thm.~18.1.3\ in  \cite{Hasselblatt-Katok}).

Choose $\epsilon = \epsilon(\eta)$ according to Lemma \ref{lem:dubound}, and such that $\epsilon<\mathfrak{h} := \mathfrak{h}(M,L_i)$,  $i=0,1$. 
Let $k\in \N$, $u\in\mathcal{M}_{J_{kt}}(\hat{x}_-, \hat{x}_+,L_0,\psi^{-k}(L_1))$ such that {$E(u)< \epsilon$}, and assume that $\mathcal{S}^k(x_+(0))$ is a $(q,p,\delta)$-chord, the case that $\mathcal{S}^k(x_-(0))$ is a $(q,p,\delta)$-chord is treated similarly.
For $s\in \R$, we set $y^s := \hat{u}(s,0)\in L_0$  
and consider the sequence $(y^s_n)_{n\in \Z}$ in $M$, given by 
\begin{align*}
y^s_n := \begin{cases} \hat{u}(s,n) &\text{ if } 0\leq n \leq k,\\ 
 \psi^{n}(y^s) &\text{ if } n< 0, \\ \psi^{n-k}(\hat{u}(s,k)) &\text{ if }
n> k.\end{cases}
\end{align*}
By Lemma \ref{lem:dubound}, 
$\mathcal{Q}_s := \{y^s_n\}_{n\in \Z}$, $s\in \R$, is a continuous family of $\eta$-pseudo-orbits for $\psi$. Note that since $y^s_0 = y^s \to x_+(0)$ as $s \to \infty$, $y^s \in L_0$, $y^s_k\in L_1$, we have that $y_0^s \in W^{u}_{\delta}(q)$ and $y^s_k\in W^s_{\delta}(p)$ as $s$ sufficiently large. Since $\delta \leq \delta^*$  and  $N_{\delta}(K) \subset U$, it follows that $\mathcal{Q}_s \subset U$ for $s$ sufficiently large.
In fact,  $\mathcal{Q}_s \subset U$ for all $s\in \R$. Otherwise,  if $s_0 := \inf\{s\in \R\, |\, \mathcal{Q}_s\subset U\}>-\infty$, then  
$\mathcal{Q}_{s_0}$ intersects $\partial U$ which contradicts (1). Hence, we can find by (2)  a continuous family of orbits $(\mathcal{O}(x^s))_{s\in \R}$ in $K$ that $\delta$-shadows $(\mathcal{Q}_s)_{s\in \R}$.  It follows that $(\mathcal{S}^k(x^s))_{s\in \R}$ is a continuous family of $(q,p,2\delta)$-chords. Note that the orbit $\mathcal{O}(x_+(0))$ $\delta$-shadows the pseudo-orbit $\mathcal{Q}_s$ if $s$ is sufficiently large. By (ii), $2\delta$ is an expansivity constant, and therefore in fact $x^s = x_+(0)$ for $s$ sufficiently large. Similarly, since also $x_-(0) \in K$, we have that $x^s=x_-(0)$ for $s$ sufficiently small.   
Since $\psi^k(W^u_{2\delta}(q))$ intersects $W^s_{2\delta}(p)$ transversally, the starting points of $(q,p,2\delta)$-chords of length $k$ are isolated. It follows that the path $(x^s)_{s\in \R}$ is constant, so $x_-(0) = x_+(0)$. Since $E(u) < \epsilon<\mathfrak{h}$, it follows that $E(u) = 0$, and we conclude that  $u$ is constant.

This concludes the statement in the non-degenerate case. To pass over to a degenerate pair $(L_0,L_1)$, note the following.  Since the iterations of the local unstable manifold already intersect the local stable manifold transversally, it is possible to perturb $L_0$ and $L_1$ outside the segments $W_{\delta}^s$ resp.\ $W_{\delta}^u$ in the $C^{\infty}$ topology to a non-degenerate admissible pair. 
Since $\epsilon>0$ can be chosen to be constant in a neighbourhood of $(L_0,L_1)$, the assertion of the Theorem also holds for the degenerate pair $(L_0,L_1)$.
\end{proof}

\section{Proof of Theorem \ref{thm:barcode_lagr}: Links in Horseshoes and crossing energy}\label{sec:Theorem3}
\subsection{Restricting the Floer chain complex to  link complements.}\label{sec:prelim_surface} 
We now discuss restrictions of the Lagrangian Floer chain complexes to link complements if $M$ is a surface. By passing to the suspension flow $\varphi$ on $Y=S^1 \times M$, every finite union of periodic orbits $\mathcal{O}\subset M$ of $\psi$, that is, a finite invariant set of $\psi$, yields a link (or braid)  $\mathcal{L}= \mathcal{L}_{\mathcal{O}}\subset Y$ of periodic orbits of  $\varphi$, and vice versa,  see \S \ref{sec:mappingcyl}.  
For such a set $\mathcal{O}\subset M$ and two closed embedded curves $L_0, L_1\subset M\setminus \mathcal{O}$, we write $\Lambda_i := \iota(L_i)$, $i=0,1$, where $\iota:M \to Y, x \mapsto (0,x)$, and 
 denote by  $\mathcal{Z}^+(Y,\mathcal{L}_{\mathcal{O}},\Lambda_0,\Lambda_1)$ the set of  homotopy classes of paths from $\Lambda_0$ to $\Lambda_1$ in $Y\setminus \mathcal{L}_{\mathcal{O}}$ that are everywhere transverse to the horizontal surfaces $\{t\} \times M$. Two  isotopies $(L^t_i)_{t\in [0,1]}$, $i=0,1$,  supported in  $M\setminus \mathcal{O}$ from $L_i$ to $L_i'$ define a bijection 
\begin{align*}\Upsilon = \Upsilon^{L_1,L'_1}_{L_0,L'_0}:\mathcal{Z}^+(Y,\mathcal{L}_{\mathcal{O}},\Lambda_0,\Lambda_1)\to \mathcal{Z}^+(Y,\mathcal{L}_{\mathcal{O}},\Lambda'_0,\Lambda'_1), \end{align*}
by sending a path $\hat{\eta}:[0,R] \to Y$ to the glued path $(\iota \circ \overline{\alpha})\#\hat{\eta}\# (\iota \circ \beta)$, where $\alpha$ resp.\  $\beta$ are paths in $M$ obtained by restricting the isotopy of $L_0$ resp.\ $L_1$ to the point $\eta(0)\in L_0$ resp.\ $\eta(1) \in L_1$, and $\overline{\alpha}$ is the reverse path of  $\alpha$.  
If $\mathcal{O} \neq \emptyset$ 
 and $L_0, L_1$ are non-contractible in $M\setminus \mathcal{O}$, then the map $\Upsilon^{L_1,L'_1}_{L_0,L_0'}$ does not depend on the choice of isotopies. Write $\Upsilon_{L_0,L'_0} = \Upsilon^{L_1,L_1}_{L_0,L'_0}$ and $\Upsilon^{L_1, L'_1} = \Upsilon^{L_1,L'_1}_{L_0,L_0}$. Then,   $\Upsilon^{L_1,L'_1}_{L_0,L'_0} = \Upsilon^{L_1,L'_1}\circ \Upsilon_{L_0,L'_0}$.

Let $(L_0,L_1)$ be an admissible pair. Note that if $M\neq S^2$, then any closed non-contractible curve $L$ is monotone, and if $M=S^2$, then any  equator $L$ is, which means that $L$ divides $S^2$ into two components of equal area. 
In the special situation of surfaces, we can define chain complexes  $(\CF^{\rho}(L_0,\psi^{-k}(L_1)),\partial^{\rho})$ associated to an element $\rho\in \mathcal{Z}^+(Y,\mathcal{L}_{\mathcal{O}},\Lambda_0,\Lambda_1)$.
This construction is analogous to that of Legendrian contact homology in link complements in \cite{Alves-thesis}, see also \cite[\S 5]{AlvesPirnapasov}, although arguments simplify in the setting of Lagrangian Floer homology. We also refer to \cite{BraidFloer} (Hamiltonian Floer homology in braid complements) and \cite{Momin} (cylindrical contact homology in link complements) for related homology  theories.
For  a pair $\hat{x} = (x,[v_x])\in \widetilde{\mathcal{P}}_k$, we use the notation $\hat{x}\in \rho$ if the curve $t\mapsto ([kt],\phi^{kt}_H(x)) \in S^1 \times M$, $t\in [0,1]$, represents $\rho$. 
Let $\CF^{\rho}(L_0,\psi^{-k}(L_1))\subset\CF(L_0,\psi^{-k}(L_1))$ be the subspace generated by $\hat{x} \in \rho$,  and define $\partial^{\rho}:\CF^{\rho}(L_0,\psi^{-k}(L_1)) \to \CF(L_0,\psi^{-k}(L_1))$ in the same way as $\partial$ with the additional requirement that for  
the holomorphic curves $u$ that appear in the definition, the map $\widetilde{u}$ does not intersect the union of the trivial cylinders over $\mathcal{L}_{\mathcal{O}}$, that is, it does not intersect the set 
$\mathrm{Cyl}_{\mathcal{O}}:=\{(s,t, \phi_H^t(p)) \, |\, s\in \R, t\in S^1, p\in \mathcal{O}\}\subset \R \times S^1\times M$. We note that the lift  $\widetilde{\mathrm{Cyl}_{\mathcal{O}}}$ of  $\mathrm{Cyl}_{\mathcal{O}}$ to $\R \times \R \times M$ is the image of a finite collection of $\widetilde{J}$-holomorphic curves.  In order that the definition of $\partial^{\rho}$ makes sense, one has to verify that the relevant moduli spaces, when restricting to  only those curves $u$ above, are compact in the $C^{\infty}_{\loc}$ topology. This follows from the $C^{\infty}_{\loc}$-compactness of the moduli-spaces without restriction and the positivity and stability of intersections of $\widetilde{J}$-holomorphic curves, see \cite{AlvesPirnapasov} for details of the argument.  
Note that for curves $u$ relevant for the definition of $\partial^{\rho}$, the map  $s\mapsto \overline{u}(s,t)$ provides a homotopy of its asymptotics in $Y\setminus \mathcal{L}_{\mathcal{O}}$ relative to $(\Lambda_0,\Lambda_1)$, hence in fact $\im (\partial^{\rho})\subset \CF^{\rho}(L_0,\psi^{-k}(L_1))$. 
Similarly, by using again the positivity of intersection property of $\widetilde{J}$-holomorphic cylinders, the argument that $\partial^2= 0$ carries over and one can show that $(\partial^{\rho})^2 = 0$. We denote the resulting homology by  $\HF^{\rho}(L_0,\psi^{-k}(L_1))$. The homology is invariant under Hamiltonian isotopy $L^t_0 = \phi^t_F(L_0)$, $t\in [0,1]$, from $L_0$ to $L'_0 = L^1_0$, as long as it is supported in $M\setminus \mathcal{O}$, which means that $\phi^t_F$ fixes $\mathcal{O}$ for all $t\in [0,1]$: restricting the count of holomorphic curves $u$ to curves such that $\widetilde{u}$ is disjoint from $\mathrm{Cyl}_{\mathcal{O}}$, defines a map $f_0^{\rho}:\CF^{\rho}(L_0,\psi^{-k}(L_1)) \to \CF^{\Upsilon(\rho)}(L_0',\psi^{-k}(L_1))$, where $\Upsilon= \Upsilon_{L_0,L_0'}$. And one can show, again using positivity and stability of intersections, that  $f^{\rho}$ is a chain map. With  $g_0^{\rho}:\CF^{\Upsilon(\rho)}(L'_0,\psi^{-k}(L_1)) \to \CF^{\rho}(L_0,\psi^{-k}(L_1))$ similarly defined,  $g_0^{\Upsilon(\rho)} \circ f_0^{\rho}$ is chain homotopic to the identity. Hence $f_0^{\rho}$ induces an isomorphism 
\begin{align}\label{iso_sigma}
\HF^{\rho}(L_0,\psi^{-k}(L_1)) \overset{\cong}{\longrightarrow} \HF^{\Upsilon(\rho)}(L'_0,\psi^{-k}(L_1)).
\end{align} 
Similarly, for a Hamiltonian isotopy $L^t_1$, $t\in [0,1]$, from $L_1$ to $L'_1 = L^1_1$ that is supported in $M\setminus \mathcal{O}$, we obtain an isomorphism 
\begin{align}\label{iso_sigma2}
\HF^{\rho}(L_0,\psi^{-k}(L_1)) \overset{\cong}{\longrightarrow} \HF^{\Upsilon(\rho)}(L_0,\psi^{-k}(L'_1)),
\end{align} 
where now $\Upsilon = \Upsilon^{L_1, L_1'}$. 
Finally, the above isomorphisms behave naturally under composition and are independent of the choice of isotopy.

\subsection{Horseshoes for surface diffeomorphisms and crossing energy}\label{sec:horseshoes_link_crossing}
In this section we discuss specific horseshoes $K$ in surface diffeomorphisms and obtain a uniform energy estimate of some Floer strips whose suspension crosses a link induced by certain horseshoe orbits in $K$.

Let  $\psi: M \to M$ be a surface  diffeomorphism with $h_{\topo}(\psi)>0$. Then, by Katok's theorem \cite[Thm.~S.5.9]{Hasselblatt-Katok}, for any $e>0$ with $e<h_{\topo}(\psi)$, there exists a locally  maximal hyperbolic set  $K$ such that \begin{align}\label{eq:entropy}
h_{\topo}(\psi|_{K}) \geq h_{\topo}(\psi)-e.\end{align}
We now discuss how one can choose the sets $K$ more specifically.
By the construction of Katok and Mendoza  \cite[Thm.~S.5.9]{Hasselblatt-Katok}, there exists a collection of finitely many pairwise disjoint rectangles $D_1,\ldots, D_L$ in $M$ such that $K$ is obtained as the union $\bigcup_{m=0}^{N-1}\psi^m(K')$, where $K'=\bigcap_{j\in \Z} \psi^{Nj}({U})$ is the $\psi^N$-invariant set induced from the isolating neighbourhood $U= \bigcup_{j=1}^L{D}_j$. 
We write the rectangles $D_j$ as images of embeddings $\iota_j:[0,1]^2 \to M$, and the boundaries as $\partial D_j = v^j_-\cup v^j_+ \cup h^j_- \cup h^j_+$ with  $v^j_-:= \iota_j(\{0\}\times [0,1])$, $v^j_+:= \iota_j(\{1\}\times [0,1])$, $h^j_-:= \iota_j([0,1]\times \{0\})$, and $h^j_+:= \iota_j([0,1]\times \{1\})$. 
Denote by $\Sigma = \{\underline{a} = (a_j)_{j\in \Z}, a_j\in \{\Omega_1, \ldots, \Omega_L\}\}$ the set of bi-infinite sequences in $L$ symbols $\{\Omega_1,\ldots, \Omega_L\}.$ For any $i\neq j$,  $\psi^{N}(D_j)$ intersects $D_i$, and for any element $\underline{a}= (a_j)_{j\in \Z} = (\Omega_{i_j})_{j\in \Z} \in \Sigma$, there exists a unique point $\pi(\underline{a})=\bigcap_{j\in\Z} \psi^{-jN}(D_{i_j})\in K'$ which defines a continuous map $\pi:\Sigma \to K'$ with $\pi \circ \sigma = \psi^N|_{K'} \circ \pi$. 
Moreover, 
 for a sequence $\underline{a} = (a_{j})_{j\leq 0} = (\Omega_{i_j})_{j\in\Z}$, $$l_0 = l_0(\underline{a}) := \bigcap_{j\leq 0} \psi^{-jN}(D_{i_j})$$  is a segment in $D_{i_0}$ that connects $v^{i_0}_{\pm}$ without  intersecting $h^{i_0}_{\pm}$ and is contained in  a local unstable manifold of a point $q\in K'$. Similarly,  for a sequence $\underline{b} = (b_{j})_{j\geq 0} = (\Omega_{i'_j})_{j\in \Z}$,   $$l_1 = l_1(\underline{b}) = \bigcap_{j\geq 0} \psi^{-(j-1)N}(D_{i'_j})$$ is a segment in $\psi^N(D_{i'_{0}})$ that connects $\psi^N(h^{i'_{0}}_{\pm})$ without intersecting $\psi^N(v^{i'_{0}}_{\pm})$ and is contained in a local stable manifold of some point $p \in K'$.
Fix $\underline{a} = (a_{j})_{j\leq 0}$ and $\underline{b} = (b_{j})_{j\geq 0}$. Write $a_{0} = \Omega_{i_0}$,  $D^0:=D_{i_0}$, $b_0 = \Omega_{i'_0}$, $D^1:=D_{i'_0}$.   For any $(n-2)$-tuple  $(i_1,\ldots, i_{n-2})$, $n\geq 3$,  consider the unique point $x=x_{(i_1,\ldots, i_{n-2})}\in M$,   
$$x := l_0 \cap  \bigcap_{j=1}^{n-2} \psi^{-jN}(D_{i_j}) \cap \psi^{-(n-1)N}(\psi^{-N}(l_1)) = \pi(\underline{c}),$$ 
where $\underline{c} = (c_j)_{j\in \Z}$ is given by  
\begin{align}
c_j = \begin{cases} a_j,  &\text{ if } j\leq 0, \\                                       \Omega_{i_j}, &\text{ if } 1\leq j\leq n-2, \\
b_{j-(n-1)},  &\text{ if } j\geq n-1.
    \end{cases}
\end{align}
Let $C_{(i_1,\ldots, i_{n-2})}:[0,nN] \to S^1 \times M$ be  the $\varphi$-chord from $\lambda_0 = \{0\} \times l_0$ to $\lambda_1 = \{0\} \times l_1$ given by $C_{(i_1,\ldots, i_{n-2})}(t):=\varphi^t(x_{(i_1,\ldots,i_{n-2})})$.

As we obtained in \cite{Meiwes_horseshoes} (in a more general setting),  by choosing $K$ and $U= \bigcup_{j=1}^L D_j$ suitably we can assume that additionally  
\begin{enumerate}[(I)]
\item  $\psi^m({U}) \cap {U} = \emptyset$, for all $m=1, \ldots, N-1, N+1,\ldots, 2N-1$; 
\item $\psi^N(v^j_{\pm})\cap D_i = \emptyset$, for all $i,j\in \{1, \ldots, L\}$;
\item  either $\psi^{2N}(v^i_{\pm}) = h^i_{\pm}$ or $\psi^{2N}(v^i_{\pm}) = h^i_{\mp}$, for all $i\in\{1, \ldots, L\}$.
\end{enumerate}
By (III), for each $i=1, \ldots, L$, we can glue $D_i$ to the two surfaces $E^i_{\pm} = \{([t], \varphi^t(x)\, |\, t\in [0,2N], x\in v^i_{\pm}\}$ and obtain a piecewise immersed pair-of-pants $F_i$ in $S^1 \times M$. By (I) and (II), the surfaces $F_i$,  $i=1, \ldots, L$, are in fact (piecewise) embedded and pairwise disjoint.  The boundary of $F_i$ is a link $\mathcal{L}_i = \mathcal{L}_{\mathcal{O}_i}$, a lift of a collection of three periodic orbits $\mathcal{O}_i$ given by iterates of the corners of $D_i$ (there are two corners of $D_i$ that belong to the same orbit).  We equip  $F_i$ with the orientation induced by the vector field of $\varphi$ along $D_i$.  
Let $\mathcal{O}= \bigcup_{i=1}^L \mathcal{O}_i$, $\mathcal{L}=\mathcal{L}_{\mathcal{O}}= \bigcup_{i=1}^L \mathcal{L}_i$. Fix $l_0$ and $l_1$ as above.  
\begin{lem}\label{lem:choice_of_L}
Assume $L_0$ and $L_1$ are closed curves in $M\setminus \mathcal{O}$ such that $L_0 \cap D^0 = l_0$, $L_1 \cap \psi^N(D^1) = l_1$. Then, for any tuple  $(i_1,\ldots, i_{n-2})$, the chord $C_{(i_1,\ldots, i_{n-2})}$ is the unique $\varphi$-chord in its homotopy class of paths from $\Lambda_0 = \{0\} \times L_0$ to  $\Lambda_1 = \{0\} \times L_1$ in $Y\setminus \mathcal{L}$ relative to $(\Lambda_1, \Lambda_2)$. 
\end{lem}
\begin{proof}
Let $\hat{C}:[0,nN] \to Y$ be a $\varphi$-chord from $\Lambda_0$ to $\Lambda_1$ homotopic to $C_{(i_1,\ldots, i_{n-2})}$ in $Y\setminus \mathcal{L}$ relative to $(\Lambda_1,\Lambda_2)$. Then there exists $0\leq t_0 < \cdots< t_{n} \leq nN$ with $\hat{C}(t_i) \in F_{i_j}$ and such that 
for all $0\leq m\leq n-1$ with $i_{m} = i_{m+1}$ the path $\hat{C}|_{[t_m,t_{m+1}]}$ is not homotopic relative to its endpoints to a path contained inside $F_{i_m}$. 
Hence, by (I)-(III) above, for any $0\leq m \leq n-1$,  $|t_{m+1} -t_m| \geq N$. Hence $t_{m} = mN$, and $\hat{C} = C_{i_1, \ldots, i_{n-2}}$ by construction. 
\end{proof}
Fix a non-degenerate, admissible pair $(L_0,L_1)$ that satisfies the assumptions of Lemma \ref{lem:choice_of_L}. Additionally assume that the algebraic intersection number of the curves $\Lambda_0 = \{0\} \times L_0$ and $\Lambda_1 = \{0\}\times L_1$  with any of the surfaces $F_1, \ldots, F_L$ is zero. This holds for example if $L_i$, $i=0,1$, intersect  those surfaces only in $U$ and $\psi^N(U)\cap \bigcup_{i=1}^L F_i$, and intersect the connected components of $U$ and $\psi^N(U)$ in segments that connect opposite edges.     
Let $\mathcal{Z}^+_1(l_0,l_1) \subset \mathcal{Z}(Y, \mathcal{L}_{\mathcal{O}}, \Lambda_1, \Lambda_2)$
be the set of homotopy classes that can be represented by the chords $C_{(i_1,\ldots, i_{n-2})}$ of $\varphi$ of the form above.

\begin{prop}\label{prop:crossing_energy}
Let $(L'_0,L'_1)$ be a non-degenerate admissible pair of curves in $M\setminus \mathcal{O}$ such that  $L'_i$ and $L_i$ are Hamiltonian isotopic in $M \setminus \mathcal{O}$, $i=0,1$. Assume $J=J_t$ is of the form  \eqref{specialJ}. 
Let $\Upsilon=\Upsilon^{L_1, L'_1}_{L_0,L'_0}$ be defined as in \S  \ref{sec:prelim_surface}. There is $\varepsilon>0$ such that if $u\in \mathcal{M}_{J_{kt}}(\hat{x}_-, \hat{x}_+, L'_0,\psi^{-k}(L'_1))$,  $k\in \N$, satisfies   
\begin{enumerate}[(1)]
\item $\hat{x}_-\in \Upsilon(\rho)$ or $\hat{x}_+\in \Upsilon(\rho)$ for some $\rho \in \mathcal{Z}^+_1(l_0,l_1)$, and
\item $\im (\overline{u}(s,\cdot)) \cap \mathcal{L}_{\mathcal{O}} \neq \emptyset\, $  for some $s\in \R$, 
\end{enumerate}
then 
$$E(u)\geq \varepsilon.$$
\end{prop}
\begin{proof}
Given a homotopy class $\rho$ of paths from $\Lambda_0 =\{0\}\times L_0$ to $\Lambda_1 = \{0\}\times L_1$ in $Y\setminus \mathcal{L}$, denote by $\kappa(\rho)$ 
 the maximal algebraic intersection number of a path in ${\rho}$  with $\bigcup_{i=0}^L F_i$. Here we naturally only consider transverse paths to $\bigcup_{i=0}^L F_i$ and also count intersections at the endpoint of the paths. By our assumption that the algebraic intersection number of $\Lambda_i$, $i=1,2$, with any of the surfaces $F_1, \ldots, F_L$ vanishes, the number  $\kappa(\rho)$ is well defined. 
Analogously define $\kappa'(\rho')$ 
for any homotopy class $\rho'$ of paths from $\{0\} \times L_0'$ to $\{0\} \times L'_1$ in $Y\setminus \mathcal{L}$. Fix some $\chi=\chi(L_0', L_1') \geq 0$ such that for any  $\rho \in \mathcal{Z}^+(Y,\mathcal{L}, \Lambda_0,\Lambda_1)$,
\begin{align}\label{eq:eta00} \kappa'(\Upsilon(\rho)) \geq \kappa(\rho) - \chi.
\end{align}

 \color{black}

Let $B:=4N(\chi +3)$. Given $\eta>0$, we set $D^{\eta}_i:=\{x\in M\, |\, \dist(x,D_i) < \eta\}$, $1\leq i\leq L$.  By (I) and (II) we can choose $\eta>0$ such that \begin{enumerate}[(a)] 
\item $\eta < \dist(D^{\eta}_i,\psi^n(D^{\eta}_j))$,  for all  $n=1, \ldots, N-1, N+1, \ldots, 2N-1$ and $i,j\in\{ 1,\ldots, L\}$;
\item $\eta < \dist(D^{\eta}_i,D^{\eta}_j)$, for all $i\neq j\in \{ 1,\ldots, L\}$;
\item $\eta < \dist(\phi_H^{t+n}(v^i_{\pm}), \phi_H^{t}(v^j_{\pm}))$,  for all     $t \in [0,2N],  n\in \Z\setminus \{0\}$, $t+ n \in [0,2N)$, and $i,j\in\{1,\ldots, L\}$;
\item $\eta< \dist(\phi^N_H(v^{i}_{\pm}), D_j)$, for all $i, j\in \{1,\ldots, L\}$;
\item $\eta < \dist(\phi_H^{t}(v^i_{\pm}), \phi_H^{t}(v^j_{\pm}))$,  for all $t \in [0,2N]$,  $i\neq j\in\{1,\ldots, L\}$;
\color{black}
\end{enumerate}
Choose $\varepsilon=\varepsilon(\eta,B)>0$ according to Corollary~\ref{cor:pseudo_2}.
We assume the contrary and let $u\in \mathcal{M}_{J_{kt}}(\hat{x}_-,\hat{x}_+, L_0',\psi^{-k}(L'_1))$, $k\in \N$,  
 such that $(1)$ and $(2)$ hold, but $E(u)<\varepsilon$.
Assume that $\hat{x}_+ \in \rho'=\Upsilon(\rho)$ for some $\rho = [C_{(i_1,\ldots, i_{n-2})}] \in \mathcal{Z}^+_1$, the argument for the other case is analogous. 
Let $s_0\in \R$ be the minimal $s\in \R$ such that the path $t\mapsto \overline{u}(s,t)$ intersects $\mathcal{L}$. 
We consider the path $$\gamma:[0,k]=[0,nN]\to M, \quad \gamma(t) = \overline{u}(s_0,t).$$ 
Note that  $\gamma$  intersects $F_1, \ldots, F_L$ in at least $l:=n-\chi$ many points
$\gamma(t_1),$ $\ldots,$ $\gamma(t_l)$, $0\leq t_1<\cdots<t_l\leq k$. Moreover,  we can choose such points so that 
if $\gamma(t_i)$ and $\gamma(t_{i'})$, $i<i'$, lie in the same surface $F_j$, $j=1, \ldots, L$,  then the path $\gamma_{[t_i,t_{i'}]}$ is not homotopic relative to its endpoints to a path contained in $F_j$.

\begin{claim}\label{cl:F-D}
Let $1\leq i\leq l$. Assume that $\gamma(t_i) \in F_j/D_j$ for some $j\in \{1,\ldots, L\}$. Then, there is $t^{-}_i\in \Z$ with $t^-_i < t_i< t^-_i + 2N =:t^+_i$ such that $\gamma(t^-_i), \gamma(t^+_i) \in D^{\eta}_j$ (provided that  $t^-_i,t^+_i \in [0,k]$). Moreover,  $t_{i'} \notin [t^-_i,t^+_i]$ for all $i'\neq i$. 
\end{claim}
\color{black}
\begin{claimproof}
We write $\gamma(t_i)=  \varphi^{r}(x)$ with $x\in v_+^j \cup v_-^j$, $0\leq r<2N$, and put $t^-_i := t_i - r$, $t^+_i :=t^-_i + 2N$. 
By Corollary \ref{cor:pseudo_2}, for all $t\in [t^-_i,t^+_i]\cap[0,k]$ we have that  $\overline{d}(\gamma(t), \varphi^{t -t_{i}}(\gamma(t_i))) < \eta$, and hence 
\begin{align}\label{eta11}
\overline{d}(\gamma(t), \varphi^{t-t^-_i}(x)) < \eta.
\end{align}
This means that $\gamma(t^-_i), \gamma(t^+_i) \in D^{\eta}_j$ (provided that  $t^-_i,t^+_i \in [0,k]$). To show the second assertion, assume by contradiction that there is $i'\neq i$ such that $t_{i'}\in[t^-_i, t^+_i]$. With  \eqref{eta11}, it follows from (a), (b), and (d) that $\gamma(t_{i'}) \notin D_{j'}$ for all $j'\neq j$, and  it follows from (a), (c), (d), and (e) that $\gamma(t_{i'}) \notin F_{j'}\setminus D_{j'}$ for all  $j'\neq j$. Hence $\gamma(t_{i'}) \in F_j$. With \eqref{eta11}, it follows now from (a), (c), and (d)  that 
 $\gamma(t_{i'}) \in \varphi^{t_{i'} - t^-_i}(v^j_+ \cup v^j_-)$. Therefore, $\gamma|_{I}$ is  homotopic relative to its endpoints to a path in $F_j$, where $I$ is the interval with endpoints $t_i$ and $t_{i'}$. This is a contradiction.
\end{claimproof}

Let $1\leq i\leq l$. If $\gamma(t_i) \in F_j\setminus D_j$ for some $j\in \{1, \ldots, L\}$, we choose $t^{\pm}_i$ as in Claim \ref{cl:F-D}, and if $\gamma(t_i) \in D_j$ for some $j\in \{1, \ldots, L\}$,  we set $t^{-}_i = t^+_i :=  t_i\in \N$.
\begin{claim}\label{cl:D}
Let $1\leq i<i'\leq l$. Then $t^-_{i'} -t^+_i\geq N$.
\end{claim}
\begin{claimproof}
By the second assertion of Claim \ref{cl:F-D} we have  $t^-_{i'} > t^+_i$.  For any $n\in\{1, \ldots, N-1\}$ with $n+t^+_i\leq k$, it holds by Corollary~\ref{cor:pseudo_2} that   $\overline{d}(\gamma(n+t^+_i), \varphi^n(\gamma(t^+_i))) < \eta$, and hence by~$(a)$, $\gamma(n+t^+_i) \notin \bigcup_{j=1}^L D^{\eta}_j$. The claim follows by the first assertion of Claim \ref{cl:F-D}.
\end{claimproof}

We write $l=T_+ + T_-$, where  $T_+$  is  the number of $i=1,\ldots, l$ with  $\gamma(t_i) \in \bigcup_{j=1}^L D_j$, and  $T_-$ the number of
$i=1, \ldots, l$ with $\gamma(t_i) \in \bigcup_{j=1}^L F_j\setminus D_j$. 
By Claims \ref{cl:F-D} and \ref{cl:D}, 
$Nn = k\geq NT_+ + 3NT_- - 4N$. 
By \eqref{eq:eta00},  $T_+ + T_- \geq n-\chi \geq  T_+ + 3T_- -4- \chi$, hence $2T_- \leq \chi+4$. 
It follows that  
\begin{align}\label{eq:T_+geq}
T_+ \geq n-\frac{3\chi}{2} -2 \geq n-2\chi-2.
\end{align}
Choose now $\hat{t}\in [0,k]$ with $\gamma(\hat{t}) \in \mathcal{L}$, and let $I\subset [0,k]$ be any interval with $\hat{t}\in I$ and length $|I| = 2N(2\chi + 5)  = B-2N$.  
\begin{claim}\label{cl:L_0L_i}
If $\gamma(\hat{t}) \in \mathcal{L}_j$ for some $j\in 
\{1, \ldots, L\}$, then the number of $i\in \{1, \ldots, l\}$ with $t_i \in I$ and $\gamma(t_i) \in \bigcup_{j=1}^L D_j$ is smaller or equal than $2\chi+5$. 
\end{claim}
\begin{claimproof}
Write $\gamma(\hat{t}) = \varphi^r(x)$ with $x\in \mathcal{O}\cap D_j$, $0\leq r< 2N$, $j\in \{1, \ldots, L\}$, and write $t_* := \hat{t}-r$. By Corollary \ref{cor:pseudo_2}, $$\overline{d}(\gamma(t + r + t_* ), \varphi^{t+r}(x)) = \overline{d}(\gamma(t+\hat{t}),\varphi^t(\gamma(\hat{t}))) < \eta,$$
for all $t\in [-B, B]$ with $t+ \hat{t} \in [0,k]$. 
Note that $\varphi^q(x) \subset v^j_{+} \cup v^j_{-}$ if and only if $q\in 2N\Z$. Hence, by $(a)$ and $(d)$, the points $\gamma(t+r + t_*)$ for $t\in[-B, B]$ can only possibly lie in $\bigcup_{i=1}^LD_j$ if $t+r \in 2N\Z$. The claim follows. 
\end{claimproof}

We are now in the situation to finish the proof of the proposition. 
By Claim~\ref{cl:L_0L_i} and by  \eqref{eq:T_+geq}, the number $\widehat{T}$ of $i\in \{1, \ldots, l\}$ with $t_i\notin I$ and $\gamma(t_i) \in \bigcup_{j=1}^L D_j$ satisfies $$\widehat{T}\geq T_+-(2\chi+5)\geq n-4\chi-7.$$
On the other hand, again by Claim \ref{cl:D}, and with $[0,k]\setminus I =: I_1 \cup I_2$, 
$$\widehat{T} \leq \lceil |I_1|/N\rceil + \lceil |I_2|/N \rceil < \frac{1}{N}(k - |I|) + 2 \leq n - 4\chi-8.$$
This yields the desired contradiction.
\end{proof}
\begin{rem}\label{rem:epsilon_cont}
It follows from the proof and by the possible choices of $\epsilon$ in Lemma \ref{lem:dubound} and Corollary \ref{cor:pseudo_2}, that  $\epsilon$ as a function on $(L'_0,L'_1)$ in the above proposition can be chosen to be constant in a $C^2$-small neighbourhood of any previously given, and not necessarily non-degenerate pair $(L'_0,L'_1)$.   
\end{rem}

\subsection{Isolation in the chain complex and length of bars}
To deduce Theorem \ref{thm:barcode_lagr} from the crossing energy estimates above, we will additionally need a purely algebraic statement. It generalizes Prop.\ 3.8 in \cite{CGG_barcode}.    
To this end, fix $\mathcal{C}= \CF(L_0,\psi^{-k}(L_1))$, for some $k\in \N$. 
Consider now chain complexes  $(\mathcal{C}', \partial')$ for which $\mathcal{C}'\subset \mathcal{C}$ is a subspace generated by a subset of $\widetilde{\mathcal{P}}_k$. The remaining basis vectors generate the subspace $(\mathcal{C}')^{\perp}$ with $\mathcal{C}' \oplus (\mathcal{C}')^{\perp} = \mathcal{C}$.  
We say that  a chain complex $(\mathcal{C}',\partial')$ of this form is \textit{$\epsilon$-isolated in $(\mathcal{C},\partial)$} if the following holds:  
\begin{enumerate}[(i)]
\item if $y\in \mathcal{C}'\setminus \{0\}$,
then 
$\mathcal{A}(y) - \mathcal{A}((\partial -\partial') y)> \epsilon$;
\item if $y\in (\mathcal{C}')^{\perp}\setminus \{0\}, \zeta \in \mathcal{C}'$ with $\zeta + \partial y\in (\mathcal{C}')^{\perp}$, then $\mathcal{A}(y) - \mathcal{A}(\zeta) > \epsilon$.
\end{enumerate}
In the case that $\mathcal{C}'$ is one-dimensional, this definition coincides with the definition given in \cite{CGG_barcode}.  Also, the following reduces to the statement of Prop.\ 3.8 in \cite{CGG_barcode} if all complexes $\mathcal{C}'_i$ are one-dimensional. 
\begin{prop}\label{prop:algebaric_barcode}
Assume that $(\mathcal{C}_1',\partial'_1),\ldots, (\mathcal{C}'_p,\partial'_p)$ are chain complexes with non-zero homology that are  $\epsilon$-isolated in $(\mathcal{C},\partial)$, and for which $\mathcal{C}'_i \cap \mathcal{C}'_j = \{0\}$ if $i\neq j$. Then $b_{\epsilon}(\mathcal{C}) \geq p/2$.
\end{prop}
\begin{proof}
We give here a proof which closely follows that of Prop.\ 3.8 in \cite{CGG_barcode}. We can assume for simplicity that the actions of the generators of $\mathcal{C}$ are all different, see 
 the discussion in \cite{CGG_barcode}. 
For $i=1,\ldots, p$, let $\xi_i\in \mathcal{C}_i'$ be a vector that represents a non-zero class in the homology of $\mathcal{C}_i'$ and such that the action $\mathcal{A}(\xi_i)$ is minimal among all such vectors in $\mathcal{C}'_i$.
All $\xi_i$, $i=1,\ldots, p$, are \textit{quasi $\epsilon$-robust}: 
for any vectors $y,\tau\in \mathcal{C}$ with $\mathcal{A}(\tau) < \mathcal{A}(\xi_i)$ and $\partial y = \xi_i + \tau$, it holds that  $\mathcal{A}(y) - \mathcal{A}(\xi_i) > \epsilon$.
Indeed, assume the contrary, and let $y,\tau$ be as above such that $\mathcal{A}(y) - \mathcal{A}(\xi_i) \leq  \epsilon$. By the $\epsilon$-isolation property (ii), $y = u + v$ for a non-zero vector $u\in \mathcal{C}'_i$ and some $v\in (\mathcal{C}'_i)^{\perp}$. Also by (ii), $\mathcal{A}(\partial v) < \mathcal{A}(\xi_i)$, and hence by (i),  $\mathcal{A}(\partial_i'u - \xi_i) <\mathcal{A}(\xi_i)$. 
By the minimality property of $\mathcal{A}(\xi_i)$, 
we have that $\partial_i'u -\xi_i = 0$, contradicting  $\xi_i \notin \im \partial_i'$. 
Note that also  any $\zeta\in \mathcal{C}$ with $\mathcal{A}(\zeta) = \mathcal{A}(\xi_i)$ is quasi $\epsilon$-robust, as well as any  non-trivial linear combination of quasi $\epsilon$-robust vectors.

Let $V$ be the span of the vectors $\xi_i$, $i=1,\ldots, p$, and write $V = \ker (\partial|_V) \oplus Y$, for some $Y\subset V$. 
If now $\zeta\in Y$, then by the $\epsilon$-isolation property (i) and since $\xi_i\in\ker \partial'_i$, $i=1,\ldots, p$, it holds that $\mathcal{A}(\zeta) - \mathcal{A}(\partial \zeta) >\epsilon$. 
Hence, if $\partial \zeta$ is not $\epsilon$-robust, then there is a vector $u$ with $\mathcal{A}(u) < \mathcal{A}(\zeta)$ and $w:= \zeta + u \in \ker \partial$. Note that $w$ is quasi $\epsilon$-robust.   
One can now find by induction quasi $\epsilon$-robust vectors $w_1, \ldots, w_r$ and a subspace $Y' \subset Y$ such that $\dim Y' + r + \dim(\ker(\partial|_V))= p$ and such that $\partial( Y')$ is $\epsilon$-robust.
Since a subspace in $\im \partial $ that is generated by quasi $\epsilon$-robust vectors is $\epsilon$-robust, it follows as in \cite{CGG_barcode} that $b_{\epsilon}(\mathcal{C}) \geq \max \{ \dim Y',r + \dim(\ker(\partial|_V))\} \geq p/2$. 
\end{proof}
\begin{proof}[Proof of Theorem \ref{thm:barcode_lagr}]
By  Lemma \ref{lem:choice_of_L} and the discussion in  
 \S\ref{sec:horseshoes_link_crossing}, there exist a compact hyperbolic set $K$ with isolating neighbourhood $U=\bigcup_{i=1}^L {D}_i$, a union of periodic orbits $\mathcal{O}\subset M$ that contain the corners of $D_i$, and a choice of segments $l_0\subset D^0$, $l_1\subset \psi^N(D^1)$ for some connected components $D^0$ and $D^1$ of ${U}$ such that for any admissible  pair $(L'_0,L'_1)$ of closed curves in $M\setminus \mathcal{O}$ that intersect $D^0$ and $\psi^N(D^1)$ in  $l_0$ and $l_1$, respectively, we have 
$$\limsup_{k\to \infty} \frac{\log \# \mathcal{Z}_1^k}{k} = h_{\topo}(\psi|_K) \geq  h_{\topo}(\psi)-e,$$
where $\mathcal{Z}^k_1$ is the set of classes in $\mathcal{Z}_1^+(l_0,l_1)$  that are represented by chords of length $k$. Let us call curves $L'_0$ and $L'_1$ \textit{adapted}  if they have the intersection property above. 
Choose homotopy classes $\alpha_0$ and $\alpha_1$ of closed curves in $M\setminus \mathcal{O}$ that have representing curves $\Gamma_0$ and $\Gamma_1$ that are adapted, and such that $\{0\}\times \Gamma_i\in S^1 \times M$, $i=1,2$, have zero algebraic intersection number with any of the surfaces $F_1,\ldots, F_L$ considered in \S \ref{sec:horseshoes_link_crossing}.     
Let $(L_0, L_1)$ be any admissible pair of curves that lie in $M\setminus \mathcal{O}$ such that $[L_i] =[\alpha_i]$, $i=0,1$.  
We can find two Hamiltonian isotopies supported in $M\setminus \mathcal{O}$ from  $L_i$ to adapted curves $L'_i$, $i=0,1$. 
One way to see this is as follows.    
Write the boundary of $D^0$ as $\partial D^0 = v_{+} \cup v_{-} \cup h_{+}\cup h_{-}$ as in \S \ref{sec:horseshoes_link_crossing}. Note that the endpoints of $v_{\pm}$ and $h_{\pm}$ lie in $\mathcal{O}$, and denote by $\mathring{v}_{\pm}$, $\mathring{h}_{\pm}$ the open segments in the complement of $\mathcal{O}$. We can find an isotopy in $M\setminus \mathcal{O}$ from $L_0$ to a curve $L_0''$ that has minimal intersection (in its homotopy class) with $\mathring{v}_{\pm}$, $\mathring{h}_{\pm}$ (that is, it has one intersection with $v_-$ and with $v_+$,  and none with $h_{\pm}$), by successively cancelling bigons whose first boundary component is contained in the  curve and whose second is contained in  $\mathring{v}_+$, $\mathring{v}_-$, $\mathring{h}_+$, or $\mathring{h}_-$. A cancellation of a bigon is obtained by "moving" the boundary segment of the bigon contained in the curve over the respective arc  $\mathring{v}_+$, $\mathring{v}_-$, $\mathring{h}_+$, or $\mathring{h}_-$. 
By composing afterwards with an isotopy fixing the curve outside a neighbourhood of $D^0\setminus \mathcal{O}\subset M\setminus \mathcal{O}$ we can obtain an isotopy $\Phi:S^1 \times [0,1]\to M\setminus \mathcal{O}$ from $\Phi(\cdot,0) =  L_0$ to an adapted curve  $\Phi(\cdot,1) = L_0'$. 
To be able to extend the isotopy $\Phi$ to a Hamiltonian isotopy, it needs to be exact, that is, $(\alpha_s)_{s\in [0,1]}$, $\alpha_s := \iota|_{\frac{\partial}{\partial_s}} (\Phi^* \omega)_{(s,t)}$, needs to be a family of exact $1$-forms on $S^1$. 
One can achieve this by modifying, during each step described above, the isotopy in a small disk in $M\setminus (\mathcal{O} \cup D^0)$ that intersects the curve but not the segment of it contained in the bigon to be cancelled in that step. (We might have to divide each step above into several smaller steps, and each time choose a different small disk). After these modifications, the second endpoint $L_0'$ of that isotopy will still be an adapted curve. The Hamiltonian isotopy from $L_1$ to $L'_1$ is obtained similarly.

Since any $\rho \in \mathcal{Z}_1^k$ is represented by a unique chord, $\HF^{\rho}(L'_0, \psi^{-k}(L'_1)) \neq 0$. Assume first that $(L_0,L_1)$ is additionally non-degenerate. Then, by \eqref{iso_sigma} and~\eqref{iso_sigma2}, we have $\HF^{\Upsilon(\rho)}(L_0, \psi^{-k}(L_1)) \neq 0$.  Here $\Upsilon = \Upsilon^{L'_1, L_1}_{L'_0,L_0}$.
By Proposition~\ref{prop:crossing_energy}, there are almost complex structures and  $\epsilon>0$ such that the chain complexes $\CF^{\Upsilon(\rho)}(L_0, \psi^{-k}(L_1))$, $\rho \in \mathcal{Z}^+_1(l_0,l_1)$, are $\epsilon$-isolated in $\CF(L_0, \psi^{-k}(L_1))$. It follows then by Proposition~\ref{prop:algebaric_barcode} that 
$$\hbar(\psi;L_0, L_1) \geq \limsup_{l\to \infty} \frac{\log \# \mathcal{Z}_1^k}{k}\geq h_{\topo}(\psi) - e.$$
By Remark \ref{rem:epsilon_cont}, the inequality generalizes to degenerate pairs $(L_0,L_1)$ that satisfy the assumptions of the theorem.
\end{proof}

\section{Examples for which the strong barcode entropy is positive}\label{sec:strong}
In this last section  we prove Theorem~\ref{mainthm:strong}, which gives a class of examples for which the strong barcode entropy is positive. 
Let $(\Sigma,\sigma)$ be a closed symplectic surface of genus at least $2$, and let $I_1, I_2:\Sigma \to \Sigma$ be two commuting anti-symplectic involutions such that $\Sigma\setminus \Fix(I_1)= \Sigma'\cup \Sigma''$ has two components $\Sigma'$ and $\Sigma''$ with positive genus, and such that the fixed point set $\Fix(f)$ of $f=I_2 \circ I_1$ is non-empty and finite. Some examples can be already obtained by realizing $\Sigma$ as suitably embedded in $\R^3$ 
 and the involutions as the restrictions of reflexions on coordinate planes. 
The problem of classification of triples $(\Sigma,I_1,I_2)$ as above can be reduced 
to the corresponding problem about anti-holomorphic involutions on Riemann surfaces. Those involutions are well understood, see \cite[\S 21]{Klein1882}; for a description of all commuting pairs see \cite{Natanzon86}.

In general, the fixed point set of anti-symplectic involutions on surfaces consists of a union of isolated points or of closed isolated  circles. With our assumptions, as it is easy to see, $\Fix(I_1)$ and $\Fix(I_2)$ must both consist of a union of closed circles. Moreover, $I_1$ flips $\Sigma'$ and $\Sigma''$, whereas $I_2$ preserves each of them. It follows that $\Fix(f) = \Fix(I_1)\cap \Fix(I_2)$.  We fix some $\hat{x}\in \Fix(f)$, and denote by $S$ the connected component in $\Fix(I_1)$ that contains $\hat{x}$. 
We also write $\overline{\Sigma}' = \Sigma' \cup \Fix(I_1)$ and $\overline{\Sigma}''=\Sigma'' \cup \Fix(I_1)$ for the surfaces that one obtains by adding the boundary $\Fix(I_1)$.

Let $M=\Sigma \times \Sigma$, equipped with the symplectic form $\omega = \sigma \oplus -\sigma$, and let $L\subset M$ be the Lagrangian submanifold defined by $L=\{(x,f(x))\, |\, x \in \Sigma\}$. The surface $\Sigma$ also embeds as a diagonal into $M$, $\phi:\Sigma \to M, x \mapsto (x,x)$, and we denote this Lagrangian submanifold by $\Delta = \im(\phi)$. By the  Lagrangian neighbourhood theorem, see e.g.\ \cite{McDuffSalamon_intro}, there exist a neighbourhood $\mathcal{N} = \mathcal{N}(\Sigma_0)$ of the zero section $\Sigma_0$ in $T^*\Sigma$, a neighbourhood $V$ of the diagonal  $\Delta\subset M$, and a diffeomorphism $\Phi:\mathcal{N}\to V$ such that $\Phi^*\omega = -d\lambda$, $\Phi(x)= \phi(x)$ if $x\in \Sigma_0 \cong \Sigma$. Here $\lambda$ is the canonical Liouville $1$-form on $T^*\Sigma$. We may assume, by choosing the neighbourhoods sufficiently small and by composition with a Hamiltonian isotopy fixing $\Delta$, that every connected component of $L\cap V$ is a disk, identified via $\Phi$ with the fibre of $T^*\Sigma$ in $\mathcal{N}$ over 
 a fixed point of~$f$.\footnote{For $\mathcal{N}$ small, by the Lagrangian neighbourhood theorem applied to a fibre $T^*_x \Sigma \cap \mathcal{N}$ over some $x\in \Fix(x)$, the Lagrangian disk $\Phi^{-1}(L\cap V)$ through $(x,x)$ becomes a graph of a $1$-form over $T^*_x \Sigma \cap \mathcal{N}$. Then we can use local primitives to construct that isotopy.} Let $D$ be the connected component of  $L\cap V$ that contains $(\hat{x},\hat{x})$.

Before we give the definition of the Hamiltonian diffeomorphism $\psi$ on $M$ for the proof of Theorem~\ref{mainthm:strong}, let us  make some further topological observations. 
The diagonal embedding $\phi: \Sigma \hookrightarrow M$, $x\mapsto (x,x)$, induces an injective homomorphism $\pi_1(\Sigma, \hat{x}) \to$ $\pi_1(M,(\hat{x},\hat{x}))$, and via the composition with the 
 map $\pi_1(M,(\hat{x},\hat{x})) \to \pi_1(M;L,L)$ induced by inclusion, any~$\mathfrak{a}\in \pi_1(\Sigma,\hat x)$ gives rise to  a well-defined homotopy class $\mathfrak{c}_{\mathfrak{a}}\in \pi_1(M;L,L)$ of paths in $M$ with endpoints in $L$.  Let $\mathfrak{A}\subset  \pi_1(\Sigma,\hat x)$ be the subset of classes~$\mathfrak{a}$ that have a  representative curve in $\overline{\Sigma}'$ and that cannot be represented by a multiple of a curve parametrizing~$S$. 
\begin{lem}\label{lem:topological}
Let $x_0, x_1\in \Fix(f)$, and $\tau$ any path in $\Sigma$ from $x_0$ to $x_1$. Denote by $\mathfrak{c}_0$ the homotopy class of paths in $M$ with boundary in $L$ defined by $\hat{\tau}= \phi \circ \tau$.
 Then, for all $\mathfrak{a} \in \mathfrak{A}$ it holds that $\mathfrak{c}_0\neq \mathfrak{c}_{\mathfrak{a}}$, unless $x_0=x_1 =\hat{x}$.
Moreover, if $x_0=x_1=\hat{x}$ and $\mathfrak{c}_0= \mathfrak{c}_{\mathfrak{a}}$ for some $\mathfrak{a} \in \mathfrak{A}$,  then any homotopy of paths in $\mathfrak{c}_0$ from $\hat{\tau}$ to itself is homotopic among such homotopies to the constant homotopy. 
\end{lem}
\begin{proof}
Let $\mathfrak{a} \in \mathfrak{A}$ and choose a representative in $\mathfrak{c}_{\mathfrak{a}}$ of the form $\hat{\gamma}= \phi \circ \gamma$, where $\gamma$ is a loop in $\Sigma$ with basepoint $\hat{x}$. 
Assume that there is a homotopy $h:[0,1]^2\to M$ in $M$ relative $L$ between $h(0,t) = \hat{\tau}(t)$ and $h(1,t) = \hat{\gamma}(t)$. We obtain paths $\hat{\beta}_i:[0,1] \to L$, $i=0,1$, from $(x_i,x_i)$ to $(\hat{x},\hat{x})$, defined by  $\hat{\beta}_i(s):= h(s,i)$. We can write $\hat{\beta}_i(s) = (\beta_i(s), f(\beta_i(s)))\in M = \Sigma \times \Sigma$, $s\in [0,1]$, $i=0,1$. The concatenation of paths $\hat{\beta}_0 \# \hat{\gamma}\# \overline{\hat{\beta}_1}$ is homotopic relative its endpoints to $\hat\tau$. And projecting such a homotopy to the first as well as second component in $M=\Sigma \times \Sigma$ shows that $\tau$ is homotopic in $\Sigma$ relative its endpoints both to the concatenation of paths  $\beta_0 \# \gamma \# \overline{\beta_1}$ and to $(f\circ \beta_0) \# \gamma \# (\overline{f \circ \beta_1})$. Hence $(\overline{f\circ \beta_0}) \# \beta_0 \# \gamma \# \overline{\beta_1} \# (f \circ \beta_1)$ is homotopic to $\gamma$ relative $\hat{x}$.

We claim that the  loops $\theta_0 =(\overline{f\circ \beta_0}) \# \beta_0$ and $\theta_1 =\overline{\beta_1} \# (f \circ \beta_1)$ must be contractible relative $\hat{x}$. Indeed, first, at least one is not homotopic to a non-trivial multiple of $S$ relative $\hat{x}$, since $\gamma$ is not homotopic to such a multiple and $\pi_1(\overline{\Sigma}',\hat x)$ is free. Say that $\theta_0$ is such a loop, the other case is analogous. It follows that the image of $[\theta_0]\in \pi_1(\Sigma,\hat{x})$ under the naturally induced  map $p:\pi_1(\Sigma, \hat{x}) \to \pi_1(\hat{\Sigma}' \cup \hat{\Sigma}'',\hat{x}) = \pi_1(\hat{\Sigma}',\hat x)\star \pi_1(\hat{\Sigma}'',\hat x)$ is non-trivial, where $\hat{\Sigma}'$ and $\hat{\Sigma}''$ are the spaces given by collapsing the boundary components of $\overline{\Sigma}'$ resp.\ $\overline{\Sigma}''$ to a single point. Since $f$ flips $\Sigma'$ and $\Sigma''$, the element $p([\theta_0])$ can be written in reduced terms (with respect to some generators) as a word that has both a non-zero term in $\pi_1(\hat{\Sigma}',\hat x)$ and one in $\pi_1(\hat{\Sigma}'',\hat x)$. Hence $p([\theta_0] \cdot [\gamma]\cdot [\theta_1])$ cannot be represented as a curve in $\hat{\Sigma}'$ and therefore $[\theta_0] \cdot [\gamma]\cdot [\theta_1]= [\gamma]\in \mathfrak{A}$ cannot be represented as a curve in $\overline{\Sigma}'$, a contradiction.  

To show the first assertion of the lemma, we assume additionally, by contradiction,  that $x_i\neq \hat{x}$ for some $i\in \{0,1\}$. That means that the contractible loop $\theta:= \theta_i$ is a concatenation of an arc  $\beta$ from $\hat{x}$ to $x_i$ and the arc $\overline{f\circ \beta}$ from $x_i$ to $\hat{x}$. Hence $\beta$ is homotopic to $f\circ \beta$ relative their endpoints. 
We can lift $f$ to a homeomorphism $\widetilde{f}$ on the universal covering  $\widetilde{\Sigma}$ such that $\widetilde{f}$ fixes a lift $\widetilde{{x}}$ of $\hat{x}$. Note that ${\widetilde{f}}^2$ is a deck transformation with fixed point and hence necessarily equal to the identity. Let  $\widetilde{\beta}$ be the lift of $\beta$ starting at~$\widetilde{{x}}$. By the above, also the second endpoint of $\widetilde{\beta}$ is fixed by $\widetilde{f}$.  Therefore $\widetilde{f}$, as any orientation preserving involution of the plane with at least two fixed points, must be the identity, which 
 can be shown for example by elementary arguments involving the $\widetilde{f}$-invariant set $\widetilde{\beta} \cup \widetilde{f}(\widetilde{\beta})$. Hence $f$ is the identity, which contradicts our assumptions.

To conclude the second assertion of the lemma, let $x_0=x_1 = \hat{x}$ and assume that $\mathfrak{c}_0 = \mathfrak{c}_{\mathfrak{a}}$ for some $\mathfrak{a} \in \mathfrak{A}$. By an argument as in the previous paragraph, one sees that $\beta_0$ and $\beta_1$ are contractible loops, and hence  $\hat{\beta}_0$ and $\hat{\beta}_1$ are contractible loops. Hence we can deform the homotopy $h$ among such homotopies to a homotopy $h':[0,1]^2 \to M$ such that $h'(s,0)= h'(s,1) = \hat{x}$ for all $s\in [0,1]$. That $h'$ is homotopic fixing the boundary $h'|_{\partial [0,1]^2}$ to  $h'':[0,1]^2 \to M$,  $h''(s,t) := \hat{\tau}(t)$,  $s,t\in [0,1]$, follows from  the homotopy exact sequence for the fibration $\Omega M \hookrightarrow \Lambda M \to M$ of the free loop space $\Lambda M$ with based loop space $\Omega M$ as fibre, since $M$ is aspherical $(\pi_2(M) = 0)$ and atoroidal $(\pi_1(\Lambda M) = 0)$.
\end{proof}
\begin{proof}[Proof of Theorem \ref{mainthm:strong}]
We keep the notation from above. Fix a hyperbolic metric $g$ on $\Sigma$.  We may choose the 
neighbourhoods $\mathcal{N}$ and $V$ constructed above in such a way that for some $0<\delta < 1/3$ we have that $\mathcal{N} = \{v \in T^*\Sigma \,|\, |v|_g \leq 3\delta\}$. Here $|v|_g$ denotes the norm of $v$ in the cotangent fibre with respect to the to $g$ dual metric. In coordinates $(r,x) \in \mathcal{N}\setminus \Sigma_0$, where  $r= |v|_g\in (0,3\delta), x = v/|v|_g \in S^*_g\Sigma$, the Liouville form $\lambda$ can be written as $\lambda =r\alpha$, with $\alpha = \lambda|_{S^*_g\Sigma}$. Here $S^*_g\Sigma\subset T^*M$ denotes the unit co-sphere bundle with respect to $g$. 

Let $C>2/\delta$, and choose a smooth function $H_0$ on $\mathcal{N}$ that depends only on the $r$-coordinate in $\mathcal{N} \setminus \Sigma_0$,  and, for some $0<\epsilon<\delta$ satisfies the following: 
\begin{itemize}
    \item $H_0(r,x) = H_0(r) = - C -\epsilon$, for $0\leq r \leq  \delta-\epsilon$;
    \item $H_0(r)$ is strictly convex, for $\delta-\epsilon< r< \delta$;
    \item $H_0(r) = \frac{C}{\delta} r -2C$, for $\delta \leq r \leq 2\delta$;
    \item $H_0(r)$ is strictly concave, for $2\delta< r< 2\delta+\epsilon$;
    \item $H_0(r)=\epsilon$, for $r\geq 2\delta+\epsilon$. 
\end{itemize} 
To simplify notation, we identify $\mathcal{N}$ with $V\subset M$ and write for an element $v\in V$  instead of $v=\Phi(r,x)$ simply $v=(r,x)$. This defines $H_0:V \to \R$, which  moreover extends to a smooth function $H:M \to \R$ as 
\begin{equation*}
H(z)=\begin{cases}H_0(z), &\text{ if } z\in V, \\\epsilon, &\text{ otherwise. } \end{cases}
\end{equation*}
Every class $\mathfrak{a} \in \pi_1(\Sigma,\hat{x})$ contains a unique geodesic arc $\gamma_{\mathfrak{a}}:[0,T_{\mathfrak{a}}] \to \Sigma$, $T_{\mathfrak{a}}\geq 0$,  from $\hat{x}$ to itself, parametrized by arc length. Let $\eta_{\mathfrak{a}}:[0,T_{\mathfrak{a}}] \to S^*_g\Sigma$ be its natural lift to the unit co-sphere bundle. 
We denote by $\psi = \phi_H^1$ the Hamiltonian diffeomorphism induced by $H$. 
Let $\mathfrak{a}\in \pi_1(\Sigma,\hat{x})$ be non-trivial. If $n$ is such that $nC/\delta >T_{\mathfrak{a}}$, we obtain two distinguished  intersection points $y_{0}$ and $y_1$ in $L\cap \psi^{-n}(L)$, given by $y_i = (r_i,\eta_{\mathfrak{a}}(0))\in V$, $i=0,1$, where $r_0<r_1$ are such that $\partial_r H(r_i,x) = T_{\mathfrak{a}}/n$.  
The Hamiltonian chords  $\overline{y}_i:[0,1] \to M$, $i=0,1$, given by  $\overline{y}_i(t) := \varphi^{nt}_{H}(y_i)$, are, as chords in $M$ relative to $L$, both of class $\mathfrak{c}_{\mathfrak{a}}$. If $\mathfrak{a} \in \mathfrak{A}$, then, by the first assertion of Lemma~\ref{lem:topological},  there exists no other Hamiltonian chord in $\mathfrak{a}$. Moreover,  by the second assertion of Lemma~\ref{lem:topological}, any homotopy between $\overline{y}_0$ and $\overline{y}_1$ relative to $L$ is homotopic to one that is contained in $V$ relative to $D$.   

There is a bijection of (a continuous family of) strips between $y_0$ and $y_1$ with boundaries in $L$ and $\psi^{-n}(L)$, and (a continuous family of) homotopies relative to $L$ between $\overline{y}_0$ and $\overline{y}_1$. 
It follows that the action difference 
$\mathcal{A}(y_1,v_1) - \mathcal{A}(y_0,v_0)$ (with respect to the pair $(L,\psi^{-n}(L))$) is independent of the cappings $v_0$ and $v_1$, and is identical to $-\int v^*\omega$, where $v:[0,1]^2\to V$ is a strip in $V$ with $v(i,t) \equiv y_i$, $i=0,1$, and $v(s,0) \in D$, $v(s,1) \in \psi^{-n}(D)$ for all $s\in [0,1]$. 
To calculate that action difference, consider the path  $z :[0,1] \to V\subset  M$, $$z(s):=(r(s),\eta_{\mathfrak{a}}(T_{\mathfrak{a}})) = (r_1 + s(r_2-r_1), \eta_{\mathfrak{a}}(T_{\mathfrak{a}})),$$ from $
\psi^n(y_0)$ to $\psi^n(y_1)$ in $D$. 
Consider a strip $v:[0,1]^2 \to M$ as above with $v(s,0)=z(s)$ and $v(s,1) = \psi^{-n}(z(s))$ for all $s\in [0,1]$. Then, 
\begin{align*}
\mathcal{A}(y_1,v_1) - \mathcal{A}(y_0,v_0) &= -\int_{[0,1]^2} v^*\omega \\ 
&= -\left[\int  z^* \lambda - \int(\psi^{-n} \circ z)^* \lambda\right]  \\
&=  \int(\psi^{-n} \circ z)^* \lambda \\
&=   \int_{0}^{1} r(s)\alpha(r'(s) D\psi^{-n}(z(s))\partial_r) ds \\
&=-n \int_{r_0}^{r_1} r  \partial^2_rH(r,x)dr\\
&=-n[r   \partial_rH(r,x)]_{r_0}^{r_1} + n\int_{r_0}^{r_1}   \partial_rH(r,x)dr \\
&=-(r_1-r_0)T_{\mathfrak{a}} + n(H(r_1,x) - H(r_0,x)) \\
&> nC- T_{\mathfrak{a}}. 
\end{align*}
For the fourth equation note that  $\alpha(\frac{\partial}{\partial_t}\phi^t_H(x,r))$ is identical  to $\partial_r H(x,r)$ for all $t\in \R$. 
Note that all intersections considered above are transverse and that $(L,L)$ can be perturbed to a non-degenerate pair, keeping those intersections unchanged.  It follows that any non-trivial $\mathfrak{a}\in \mathfrak{A}$  with ${T}_{\mathfrak{a}} \leq T$ gives rise to a finite bar of length greater than $nC- T$, 
that is, 
$$b_{nC-T}(\psi;L,L) \geq \# \{\mathfrak{a} \in \mathfrak{A}\, |\, T_{\mathfrak{a}}\leq T\}=:N(T).$$
Since $\Sigma'$ has positive genus and negative Euler characteristic, the number of elements in $\mathfrak{A} \subset \pi_1(\Sigma,\hat{x})$ grows exponentially, and hence
$$N(T)\geq e^{\mu T}$$
for some $\mu>0$. 
Altogether, with $R:=\frac{C}{2}$,
$$b_{Rn}(\psi;L,L)\geq N(Rn) \geq e^{\mu R n},$$ 
and therefore 
$\mathsf{H}^R(\psi;\mathcal{L}) \geq \mu R=:E.$
\end{proof}

\bibliographystyle{plain}
\bibliography{biblio}
\end{document}